\documentclass{amsart}
\numberwithin{equation}{section}
\usepackage{amssymb}
\usepackage{amsmath}
\usepackage{amsthm}
\usepackage{amscd}
\usepackage{amsfonts}
\usepackage{ascmac}
\usepackage[usenames]{color}
\usepackage{graphics}

\theoremstyle{plain}
\newtheorem{theorem}{Theorem}[section]
\newtheorem{proposition}[theorem]{Proposition}
\newtheorem{corollary}[theorem]{Corollary}
\newtheorem{lemma}[theorem]{Lemma}

\newtheorem{remark}[theorem]{Remark}

\newcommand{\bfx}{{\mathbf x}}

\newcommand{\bfC}{{\mathbf C}}

\newcommand{\bfR}{{\mathbf R}}

\newcommand{\bari}{{\overline i}}
\newcommand{\barj}{{\overline j}}

\newcommand{\bark}{{\overline k}}
\newcommand{\barl}{{\overline \ell}}

\newcommand{\barp}{{\overline p}}
\newcommand{\barq}{{\overline q}}

\newcommand{\baralp}{{\overline \alpha}}
\newcommand{\barbet}{{\overline \beta}}

\newcommand{\barvarphi}{{\overline{\varphi}}}
\newcommand{\barpartial}{{\overline \partial}}
\def\bp{\overline{\partial}}

\newcommand{\barA}{\overline {A}}

\newcommand{\mapright}[1]{\smash{\mathop{   \hbox to 0.7cm{\rightarrowfill}}
  \limits^{#1}}}

\newcommand{\Ric}{\operatorname{Ric}}

\newcommand{\grad}{\mathrm{grad}}
\newcommand{\diver}{\mathrm{div}}

\def\tr{{\hbox{tr}}}

\def\g{{\mathfrak g}}

\newcommand{\pa}{\partial}
\newcommand{\ga}{\alpha}
\newcommand{\gb}{\beta}
\newcommand{\ka}{K\"ahler }
\newcommand{\ke}{K\"ahler-Einstein }

\newcommand{\lb}{\left (}
\newcommand{\rb}{\right )}
\newcommand{\lsb}{\left [}
\newcommand{\rsb}{\right ]}

\newcommand{\J}{{\mathcal J}}
\newcommand{\rc}{\text{Ric}}
\newcommand{\ii}{\sqrt{-1}}
\newcommand{\dd}{\text{div}}

\newcommand{\R}{{\mathbb R}}

\newcommand{\dt}{\frac{d}{dt}\bigg |_{t=0}}
\title{Hodge Laplacian and geometry of Kuranishi family of Fano manifolds}
\author{Akito Futaki, Xiaofeng Sun and Yingying Zhang}
\address{Yau Mathematical Sciences Center, Tsinghua University, Haidian district, Beijing 100084, China}
\email{futaki@tsinghua.edu.cn}
\address{Department of Mathematics, Lehigh University, Bethlehem, PA 18015, USA}
\email{xis205@lehigh.edu}
\address{Yau Mathematical Sciences Center, Tsinghua University, Haidian district, Beijing 100084, China}
\email{zhangyingying@bimsa.cn}

\date{January 4, 2023}

\begin{document}

\begin{abstract}
We first obtain eigenvalue estimates for the Hodge Laplacian on Fano manifolds, which follow from
the Bochner-Kodaira formula. Then we apply it to study the geometry of the Kuranishi family of deformations of
Fano manifolds. We show that the original K\"ahler form remains to be a K\"ahler form for other members
of the Kuranishi family, and give an explicit formula of the Ricci potential. 
We also show that our set-up gives another account for  the Donaldson-Fujiki picture. 
\end{abstract}

\maketitle

\footnotetext[1]{YYZ is supported in part by NSFC Grant No. 12141101.}

\section{Introduction}
Deformation of complex structures was initially studied by Riemann in his work of studying the Abelian functions.
Max Noether considered the complex deformation of algebraic surfaces (\cite{Kodaira86}, Preface). 
The deformation of higher dimensional compact complex manifolds has been developed intensively by a series
work of Kodaira, Spencer, Nirenberg and Kuranishi in 1950's - 1960'. The deformation problem plays an important
role in understanding the local moduli theory of complex manifolds. On the other hand, the existence of canonical metrics on
complex manifolds is an important component in understanding the structure of the moduli spaces and metrics on them.
The study of Weil-Petersson metric on moduli spaces of hyperbolic Riemann surfaces and K\"ahler-Einstein manifolds involves the
further development of the Kodaira-Spencer-Kuranishi theory on such manifolds. In this direction, one of the basic questions is whether the desired canonical metrics on the original manifold is compatible with the complex deformation. This enables us to
study the local geometry of the parameter spaces.

In this paper we study the geometry of the Kuranishi family of deformations of Fano manifolds. For this purpose we first
obtain eigenvalue estimates for the Hodge Laplacian using the formula \eqref{Boch10}, which we call the {\it Bochner-Kodaira formula},
used in the proof of the Kodaira vanishing theorem \cite{Kodaira53}.
Let $M$ be a Fano manifold of dimension $m$, $\omega$ a K\"ahler form in $2\pi c_1(M)$ and $\Ric$ the Ricci form
of $\omega$. Since both $\omega$ and $\Ric$ represent $2\pi c_1(M)$ 
there exists a real smooth function $f$, called the Ricci potential, such that
$$
 \Ric\, -\, \omega = \sqrt{-1}\partial\barpartial f.
$$
Consider the weighted Hodge Laplacian $\Delta_f = \barpartial_f^\ast\barpartial + \barpartial\,\barpartial_f^\ast$ acting on 
differential forms of type $(0,q)$
where $\barpartial_f^\ast$ is the formal adjoint of $\barpartial$ with respect to the weighted $L^2$-inner product
$\int_M (\cdot,\cdot) e^f \omega^m$.
\begin{theorem}\label{ThmA}
Let $M$ be a Fano manifold and $\Delta_f$ be the weighted Hodge Laplacian as above. \\
(1)\ \ If $\Delta_f \eta = \lambda \eta$ and $\eta \ne 0$ for a $(0,q)$-form $\eta$ then $\lambda \ge q$.\\
(2)\ \ If, in (1), $\lambda = q$ and $\eta \ne 0$ then $\nabla^{\prime\prime}\eta = 0$. In particular $\eta$ is closed, and
for $q\ge 1$, $\eta$ is exact, and indeed, it is expressed as $\eta = \frac1{q}\barpartial(\barpartial_f^\ast\eta)$.
\end{theorem}
\begin{theorem}\label{ThmB}
Let $M$ be a Fano manifold and $\Delta_f$ be the weighted Hodge Laplacian as above. \\
(1)\ \ If $\Delta_f \eta = \lambda \eta$ for a $(0,q)$-form $\eta$ and $\barpartial \eta \ne 0$ then $\lambda \ge q+1$.\\
(2)\ \ If, in (1), $\lambda = q+1$ then 
\begin{equation} \label{Boch13}
(\barpartial \eta)^\sharp := g^{i\barj}g^{i_1 \barj_1} \cdots g^{i_q\barj_q}\, \nabla_\barj\, \eta_{\barj_1 \cdots \barj_q}\, \frac{\partial}{\partial z^i} \wedge 
\frac{\partial}{\partial z^{i_1}} \wedge \cdots \wedge  \frac{\partial}{\partial z^{i_q}}
\end{equation}
is a holomorphic section of $\wedge^{q+1}T^\prime M$.
\end{theorem}
\begin{corollary}[Theorem 2.4.3 in \cite{futaki88}]\label{CorC}
Let $M$ be a Fano manifold.\\
(1)\ \ If $\Delta_f u = \lambda u$ for a non-constant complex-valued smooth function $u$ then $\lambda \ge 1$.\\
(2)\ \ If, in (1), $\lambda = 1$ then $(\barpartial u)^\sharp$ is a holomorphic vector field.
\end{corollary}
\noindent
This observation was the starting point to find the obstruction to the existence of K\"ahler-Einstein metrics in \cite{futaki83.1}
inspired by \cite{KJ-FW1}, \cite{KJ}, see also \cite{futaki87}, \cite{Futaki15}. 
Similar results for coupled K\"ahler metrics are given in section 6, see Theorem \ref{ThmH}, Theorem \ref{ThmI}
and Corollary \ref{CorJ}. 

We apply Theorem \ref{ThmA} to the study of the geometry of the Kuranishi family.
We consider the Kuranishi family described by a family of vector valued $1$-forms 
$$\varphi(t) = \sum_{i=1}^k t^i\varphi_i + \sum_{|I|\ge2}t^I\varphi_I\ \in\ A^{0,1}(T^\prime M)$$
such that 
\begin{equation}\label{Kura}
\begin{cases}
 \barpartial\varphi(t) = \frac12 [\varphi(t),\varphi(t)];\\
 \barpartial^\ast_f\, \varphi(t) = 0;\\
\text{for\ all\ }i, \varphi_i \text{\ is}\ \Delta_f\text{-harmonic.}
\end{cases}
\end{equation}
See section 3 for more detail about this description.
\begin{theorem}\label{ThmD}
Let $M$ be a Fano manifold, $\omega$ a K\"ahler form in $2\pi c_1(M)$, and $\{M_t\}$ be the Kuranishi family
of the deformation of complex structures described by \eqref{Kura}.
Then $\omega$ is a K\"ahler form on $M_t$ for any $t$.
\end{theorem}

\begin{theorem}\label{ThmE} In the situation of Theorem \ref{ThmD}, the Ricci potential of $(M_t,\omega)$ is given by
$f + \log\det(1 - \varphi(t)\barvarphi(t))$ up to an additive constant, more precisely,
\begin{equation}\label{p1}
\Ric(M_t,\omega) = \omega + \sqrt{-1}\partial_t\barpartial_t (f + \log\det(1-\varphi(t)\barvarphi(t)))
\end{equation}
where $\Ric(M_t,\omega)$ denotes the Ricci form with respect to the complex structure $J_t$ on $M_t$.
\end{theorem}
\begin{remark}
In the case when $M_0$ is a K\"ahler-Einstein manifold, Theorem \ref{ThmE} has been obtained in \cite{CaoSunYauZhang2022},
Corollary 2.1, where this was used to give a necessary and sufficient condition for the existence
of K\"ahler-Einstein metrics on small deformations of a Fano K\"ahler-Einstein manifolds.
\end{remark}

After this introduction, this paper is organized as follows. In section 2, we review the proof of Kodaira vanishing theorem,
and see that the same formula as used for it can be used to prove Theorem \ref{ThmA}. Theorem \ref{ThmB} follows from Theorem \ref{ThmA}.
We also prove Lemma \ref{K1} which is used in the proof of Theorem \ref{ThmD} in section 3. 
In section 3, we review the deformation theory of Kodaira, Spencer and Kuranishi, and prove Theorem \ref{ThmD}.
We need to show that the K\"ahler form $\omega$ is invariant under the complex structure $J_t$ of $M_t$.
This is equivalent to the vanishing of the $(0,2)$-form $\varphi\lrcorner \omega$. Using Lemma \ref{K1} we can show
that $\varphi\lrcorner \omega$ is an eigenform with eigenvalue $\frac12$. Then the vanishing of $\varphi\lrcorner\omega$
follows from Theorem \ref{ThmA}. In section 4, we compute the Ricci potential $\omega$ on $M_t$, 
and prove Theorem \ref{ThmE}. 
The preliminary computations for the proof of Theorem \ref{ThmE} can be used for general K\"ahler manifolds.
In section 5 we give an alternate account of the moment map picture of Donaldson-Fujiki using the preliminary
computations in section 4. We also give an account of another moment map picture by Donaldson \cite{Donaldson17}.
In section 6, we treat the case of coupled K\"ahler metrics, and show Theorem \ref{ThmH}, Theorem \ref{ThmI}
and Corollary \ref{CorJ}. 

\section{Bochner-Kodaira formula and eigenvalue estimates for Hodge Laplacian}
We first recall the proof of Kodaira vanishing theorem along \cite{Kodaira53}, \cite{MorrowKodaira}. 
Let $(M,g)$ be a compact K\"ahler manifold, $(L,h)$ be an Hermitian line bundle over $M$.
We use the conventions of \cite{futaki88} for the curvatures of $(M,g)$. The curvature tensor is given 
for local holomorphic coordinates $z^1, \cdots, z^m$ by
$$ R_{i\barj k \barl} = \frac{\partial^2 g_{k\barl}}{\partial z^i \partial z^\barj} - 
g^{p\barq}\,\frac{\partial g_{k\barq}}{\partial z^i}\frac{\partial g_{p\barl}}{\partial z^\barj}$$
and the Ricci tensor is
$$ R_{i\barj} = R_{i\barj}{}^k{}_k = - g^{k\barl}\,R_{i\barj k\barl} 
= - \frac{\partial^2}{\partial z^i \partial z^\barj} \log\det (g_{k\barl}).$$
The Ricci identity, which is equivalent to the definition of the curvature tensor, is expressed 
for a holomorphic vector field $X$ as
$$ [\nabla_i,\nabla_\barj] X^k = R_{i\barj}{}^k{}_\ell X^\ell$$
and dually for a $(0,1)$-form $\eta$ as
$$ [\nabla_i,\nabla_\barj]\ \eta_\bark = -R_{i\barj}{}^\barl{}_\bark\ \eta_{\barl}.$$
The curvature $\psi$ of $h$ is expressed as
$$ \psi_{i\barj} = - \partial_i \partial_\barj \log h_U$$
where $h_U = h(s_U,\overline{s_U})$ for a local non-vanishing holomorphic section $s_U$ on an open set $U$.

Let $\eta \in \Gamma(\wedge^{p,q}(L))$ be an $L$-valued smooth $(p,q)$-form. We express $\eta$ locally on $U$ as
 $$\eta = \eta_{I\barj_1 \cdots \barj_q}\ s_U\otimes dz^I \wedge dz^\barj_1 \wedge \cdots \wedge dz^\barj_q$$
 where $I$ denotes the multi-index with $p$ holomorphic indices and 
 $\eta_{I\barj_1 \cdots \barj_q}$ is alternating with respect to the indices in $I$ as well as 
 the anti-holomorphic indices $\barj_1 \cdots \barj_q$. 
 Then using the symmetry of the Christoffel symbols we have
 \begin{equation}\label{Boch1}
 (\barpartial\eta)_{I\barj_0\barj_1\cdots\barj_q}
 =(-1)^p \sum_{\beta=0}^q (-1)^\beta \nabla _{\barj_\beta} \eta_{I\barj_0\barj_1\cdots \hat{\barj}_\beta \cdots \barj_q},
 \end{equation}
 \begin{eqnarray}\label{Boch2}
 (\barpartial^\ast_L \eta )_{I \barj_1 \cdots \barj_q}&=& - (-1)^p g^{i\barj} \nabla_i^L \eta_{I\barj \barj_1 \cdots \barj_q}\\
&=&  - (-1)^p h_U^{-1} g^{i\barj} \nabla_i(h_U\ \eta_{I\barj \barj_1 \cdots \barj_q})\nonumber
  \end{eqnarray}
Taking $j=j_0$ in \eqref{Boch1} and \eqref{Boch2} we obtain
\begin{eqnarray}
(\barpartial^\ast_L\barpartial\eta)_{I\barj_1\cdots\barj_q} &=& - g^{i\barj}\nabla_i^L\nabla_\barj\eta_{I\barj_1\cdots\barj_q} \label{Boch3}\\
 &&\  +\ \sum_{\beta=1}^q (-1)^{\beta+1} g^{i\barj}\nabla_i^L\nabla_{\barj_\beta}\eta_{I\barj\barj_1\cdots\hat{\barj}_\beta \cdots \barj_q}\nonumber
\end{eqnarray}
and
\begin{equation}\label{Boch4}
(\barpartial\barpartial^\ast_L\eta)_{I\barj_1\cdots\barj_q} = - \sum_{\beta=1}^q (-1)^{\beta + 1} g^{i\barj} \nabla_{\barj_\beta} \nabla_i^L \eta_{I\barj\barj_1\cdots\hat{\barj}_\beta \cdots \barj_q}
\end{equation}
Hence, for the $\barpartial$-Laplacian $\Delta_{\barpartial}^L = \barpartial^\ast_L\barpartial + \barpartial\barpartial^\ast_L$ we obtain
\begin{eqnarray}
(\Delta_\barpartial^L\,\eta)_{I\barj_1\cdots\barj_q} &=& - g^{i\barj}\nabla_i^L\nabla_\barj\eta_{I\barj_1\cdots\barj_q} \label{Boch5}\\
&&\  -\ \sum_{\beta=1}^q (-1)^{\beta} g^{i\barj}[\nabla_i^L,\nabla_{\barj_\beta}]\eta_{I\barj\barj_1\cdots\hat{\barj}_\beta \cdots \barj_q}. \nonumber
\end{eqnarray}

In the second term on the right hand side we apply the Ricci identity to the $I$-indices and obtain
\begin{eqnarray}\label{Boch5.5}
\sum_{\alpha=1}^p\ R^k{}_{i_\alpha \barj_\beta}{}^\barl\ \eta_{i_1 \cdots k \cdots i_p \barj_1 \cdots \barl \cdots \barj_q}
\end{eqnarray}
where $k$ appears at the $\alpha$-th place in $i_1 \cdots k \cdots i_p$ and $\barl$ appears in the $\beta$-th place in
$\barj_1 \cdots \barl \cdots \barj_q$. (Later we will treat only $(0,q)$-forms and $I$ will be empty so that this term will not be considered.)

Secondly, we apply the Ricci identity to $\barj\barj_1\cdots\hat{\barj}_\beta \cdots \barj_q$, and obtain
\begin{equation}\label{Boch6}
\sum_{\beta=1}^q (-1)^\beta g^{i\barj}\ R_{i\barj_\beta}{}^\bark{}_\barj\ \eta_{I\bark\barj_1\cdots\hat{\barj}_\beta\cdots \barj_q}
+ \sum_{\beta=1}^q (-1)^\beta g^{i\barj}\ \sum_{\gamma \ne \beta} R_{i\barj_\beta}{}^\bark{}_{\barj_\gamma}
\ \eta_{I\barj\barj_1\cdots\bark\cdots\hat{\barj}_\beta\cdots \barj_q}
\end{equation}
where $k$ appears in the $\gamma$-th place in $I\barj\barj_1\cdots\bark\cdots\hat{\barj}_\beta\cdots \barj_q$. 
The second term of \eqref{Boch6} vanishes because
$$ g^{i\barj}\ R_{i\barj_\beta}{}^\bark{}_{\barj_\gamma} = R^\barj{}_{\barj_\beta}{}^\bark{}_{\barj_\gamma}$$
is symmetric in $\barj$ and $\bark$, which can be checked using the Bianchi identity, while 
$\eta_{I\barj\barj_1\cdots\bark\cdots\hat{\barj}_\beta\cdots \barj_q}$ is alternating with respect to $\barj$ and $\bark$.
In the first term of \eqref{Boch6}, using the Bianchi identity we have
\begin{equation}\label{Boch7}
g^{i\barj}R_{i\barj_\beta}{}^\bark{}_\barj = - R^\bark{}_{\barj_\beta} = - g^{\ell\bark}R_{\ell\barj_\beta}
\end{equation}
where $R_{i\barj}$ denotes the $i\barj$-component of the Ricci curvature.
Hence, changing $\ell$ in the last term of \eqref{Boch7} to $i$, \eqref{Boch6} is equal to
\begin{equation}\label{Boch8}
\sum_{j=1}^q g^{i\bark}R_{i\barj_\beta}\eta_{I\barj_1\cdots\barj_{\beta-1}\bark\barj_{\beta+1}\cdots \barj_q}.
\end{equation}

Lastly we apply the Ricci identity to $L$-component and have the curvature term $\psi$ as
\begin{equation}\label{Boch9}
\sum_{\beta=1}^q g^{i\barj}\psi_{i\barj_\beta}\eta_{I\barj_1 \cdots \barj_{\beta-1}\barj\barj_{\beta+1}\cdots \barj_q}.
\end{equation}

It follows from \eqref{Boch5.5}, \eqref{Boch8} and \eqref{Boch9} that the $\barpartial$-Laplacian \eqref{Boch5} of $\eta$
is given by the following formula which we call the {\it Bochner-Kodaira formula}
\begin{eqnarray} \label{Boch10}
(\Delta_\barpartial^L\,\eta)_{I\barj_1\cdots\barj_q} &=& - g^{i\barj}\nabla_i^L\nabla_\barj\eta_{I\barj_1\cdots\barj_q}\\
&&\  + \sum_{\alpha=1}^p\ R^k{}_{i_\alpha \barj_\beta}{}^\barl\ \eta_{i_1 \cdots i_{\alpha-1} k i_{\alpha+1}
 \cdots i_p \barj_1 \cdots \barj_{\beta-1} \barl \barj_{\beta +1} \cdots \barj_q} \nonumber \\
&&\  + \sum_{\beta=1}^q g^{i\barj}(R_{i\barj_\beta} + \psi_{i\barj_\beta})\eta_{I\barj_1 \cdots \barj_{\beta-1}\barj\barj_{\beta+1}\cdots \barj_q}.
 \nonumber
\end{eqnarray}

Suppose that $K_M^{-1}\otimes L$ is ample so that there are a \ka metric $g$ and an Hermitian metric $h$ such that 
$(R_{i\barj} + \psi_{i\barj})$ is positive definite. Then if $\eta$ is a harmonic $(0,q)$-form (so that $I$ is empty and the second term of the right hand side of
\eqref{Boch10} vanish) then taking the $L^2$-inner product of $\eta$ and 
the both sides of \eqref{Boch10} we see that $\eta$ must vanish, i.e. 
$$H^q(M,  L) = 0 \quad\text{for}\quad q > 0. $$
This is the proof of Kodaira 
vanishing theorem. Next we apply \eqref{Boch10} in the following situation.
%%%%%

%%%%%
Let $M$ be a Fano manifold of dimension $m$, and regard the first Chern class $2\pi c_1(M)$ as the K\"ahler class.
Let $\omega$ be a K\"ahler form in $2\pi c_1(M)$ and express it as
$$ \omega = \sqrt{-1} \sum_{i,j=1}^m g_{i\barj}\ dz^i \wedge dz^{\barj}$$
where $z^1, \cdots, z^m$ are local holomorphic coordinates. Since the Ricci form
$$\Ric =\sqrt{-1} \sum_{i,j=1}^m R_{i\barj}\ dz^i \wedge dz^{\barj}$$
also represents $2\pi c_1(M)$ there exists a real smooth function $f$, called the Ricci potential, such that
\begin{eqnarray} \label{Boch10.5}
 \Ric - \omega = \sqrt{-1}\partial\barpartial f, \quad\text{i.e.}\quad R_{i\barj} - f_{i\barj} = g_{i\barj}.
\end{eqnarray}
Let $L=\mathcal O$ be the trivial line bundle endowed with the Hermitian metric $e^f$. We write $\Delta_f := \Delta_\barpartial^L$ for
our choice of the Hermitian metric $e^f$ on $L$. Note that this is the same as considering the weighted volume form $e^f \omega^m$ for
$(0,q)$-forms. Instead of $\eta \in A^{0,q}(L)$, we regard $\eta \in A^{0,q}(M)$, and then, by using \eqref{Boch10.5}, the Bochner-Kodaira formula \eqref{Boch10} reads
\begin{eqnarray} \label{Boch11}
(\Delta_f\,\eta)_{\barj_1\cdots\barj_q} &=& - g^{i\barj}\nabla_{i,f}\nabla_\barj\eta_{\barj_1\cdots\barj_q}\\
&&\  + \sum_{\beta=1}^q g^{i\barj}(R_{i\barj_\beta} - f_{i\barj_\beta})\eta_{\barj_1 \cdots \barj_{\beta-1}\barj\barj_{\beta+1}\cdots \barj_q} \nonumber\\
&=& - g^{i\barj}\nabla_{i,f}\nabla_\barj\eta_{\barj_1\cdots\barj_q} + q\ \eta_{\barj_1 \cdots \barj_q}
 \nonumber
\end{eqnarray}
where 
\begin{eqnarray} \label{Boch12}
\nabla_{i,f} = \nabla_i + f_i.
\end{eqnarray}
\begin{proof}[Proof of Theorem \ref{ThmA}] Let $(\cdot,\cdot)_f$ denote the weighted $L^2$-inner product with respect to the weighted volume form 
$e^f\omega^m$. The part (1) follows from
$$ \lambda (\eta, \eta)_f = (\Delta_f \eta, \eta)_f = (\nabla^{\prime\prime}\eta, \nabla^{\prime\prime}\eta)_f + q(\eta,\eta)_f.$$
If $\lambda = q$ then $\nabla^{\prime\prime}\eta = 0$. Since $\barpartial \eta$ is the skew-symmetrization of $\nabla^{\prime\prime}\eta$
it follows that $\barpartial \eta = 0$. Moreover, since $H^{0,q}_\barpartial(M) = 0$ for $q \ge 1$ on the Fano manifold $M$, $\eta$
is exact.
This completes the proof of Theorem \ref{ThmA}.
\end{proof}

\begin{proof}[Proof of Theorem \ref{ThmB}]
Since $\barpartial\Delta_f =\Delta_f\barpartial$ we have $\Delta_f \barpartial \eta = \lambda \barpartial \eta$. Apply (1) of 
Theorem \ref{ThmA} to $\barpartial \eta$ which is non-zero by our assumption. Then $\lambda \ge q+1$, which proves (1) 
of Theorem \ref{ThmB}.

If $\lambda = q+1$ then by (2) of Theorem \ref{ThmA} we have $\nabla^{\prime\prime}\barpartial \eta =0$. 
This implies that $(\barpartial \eta)^\sharp$ is holomorphic. This proves (2) of Theorem \ref{ThmB}.
\end{proof}

\begin{remark}
Using the arguments of the proof of Theorem \ref{ThmB} one can show the following. Let $M$ be a compact
K\"ahler manifold, and $\lambda_1^{(q)}$ be the first non zero eigenvalue of the Hodge Laplacian on $(0,q)$-forms.
Then we have 
$$\lambda_1^{(q)} \le \lambda_1^{(q-1)} \le \cdots \le \lambda_1^{(0)}.$$
This is even true for the real Hodge Laplacian for $q$-forms on a compact Riemannian manifold, see e.g. \cite{TakahashiJ}.
\end{remark}

Now we turn to a lemma which is proved with similar arguments as \eqref{Boch10} and will be used in the next section.
Let $M$ be a Fano manifold and the K\"ahler form $\omega$ and the Ricci form $\Ric$ satisfy \eqref{Boch10.5} as before.
Let $\varphi = \varphi^i{}_\bark\, \frac{\partial}{\partial z^i}\otimes dz^\bark$ be a smooth section of  $T^\prime M \otimes T^{\prime\prime\ast} M$, and write
\begin{eqnarray} 
\varphi\lrcorner\,\omega &:=& \varphi^i{}_\bark dz^\bark \wedge \sqrt{-1} g_{i\barj} dz^\barj + \varphi^i{}_\barj dz^\barj \wedge \sqrt{-1} g_{i\bark} dz^\bark \label{Boch14}\\
&=&  - \sqrt{-1}(\varphi_{\barj\bark} - \varphi_{\bark\barj}) dz^\barj \wedge dz^\bark \nonumber\\
&=&- \sqrt{-1}\ \psi_{\barj\,\bark}\, dz^\barj\wedge dz^\bark, \nonumber
\end{eqnarray}
where we have put 
\begin{equation*}
\psi_{\barj\,\bark} = \varphi_{\barj\,\bark} - \varphi_{\bark\,\barj}
\end{equation*}
which is of course the skew-symmetrization of $\varphi_{\barj\,\bark}$ and to which we apply the Bochner-Kodaira formula in the next section.
We denote by $\barpartial_f^\ast$ the formal adjoint of $\barpartial$ with respect to the weighted volume form $e^f \omega^m$,
i.e. for a (0,1)-form $\alpha$, 
$$\barpartial_f^\ast \alpha = - g^{i\barj}(\nabla_i + f_i) \alpha_\barj = - \nabla_f^\barj \alpha_\barj.$$
Here we have put $\nabla_f^\barj := g^{i\barj}(\nabla_i + f_i)$, and we will keep this notation below.
We also denote by $\diver_f (X)$ the divergence
$$ d\,i(X) (e^f \omega^m) = \diver_f (X)(e^f \omega^m) $$
for a type $(1,0)$-vector field $X$, i.e. $\diver_f(X) = (\nabla_i + f_i)X^i$.

\begin{lemma}\label{K1}
Suppose 
$\barpartial_f^\ast\varphi = 0$, i.e. $\nabla_f^\bark\varphi^i{}_\bark = 0$. Then
\begin{equation}\label{Boch16}
\barpartial\barpartial_f^\ast (\varphi\lrcorner\,\omega) = \sqrt{-1} \diver_f(\barpartial \varphi) + \frac12\,\varphi\lrcorner\,\omega.
\end{equation}
\end{lemma}
\begin{proof}By the assumption $\barpartial_f^\ast\varphi = 0$ we have
\begin{eqnarray*}
\barpartial_f^\ast (\varphi\lrcorner\,\omega) &=& - \sqrt{-1} \nabla_f^\bark\,(\varphi_{\barj\,\bark} - \varphi_{\bark\,\barj}) dz^\barj\\
 &=& \sqrt{-1} \nabla_f^\bark\,\varphi_{\bark\,\barj} dz^\barj.
\end{eqnarray*}
It follows from the Ricci identity that
\begin{eqnarray}
\barpartial \barpartial_f^\ast (\varphi\lrcorner\,\omega)
 &=& \sqrt{-1} \nabla_\barl\nabla_f^\bark\,\varphi_{\bark\,\barj} dz^\barl \wedge dz^\barj \nonumber\\
 &=& \sqrt{-1} \left( \nabla^\bark\nabla_\barl\,\varphi_{\bark\,\barj} - R_\barl{}^{\bark \barp}{}_\bark\, \varphi_{\barp\barj}
 - R_\barl{}^{\bark\barp}{}_\barj\,\varphi_{\bark\,\barp}\right)\,dz^\barl \wedge dz^\barj\label{K2}\\
 & & + \sqrt{-1}\nabla_\barl( f^\bark\varphi_{\bark\,\barj}) dz^\barl \wedge dz^\barj.\nonumber
\end{eqnarray}
Note that $R_\barl{}^{\bark \barp}{}_\bark = R^\barp{}_\barl$.
Note also that $R_\barl{}^{\bark\barp}{}_\barj$ is symmetric in $\barl$ and $\barj$ while $ dz^\barl \wedge dz^\barj$ skew-Symmetric in 
$\barl$ and $\barj$, and thus the term 
$$ R_\barl{}^{\bark\barp}{}_\barj\,\varphi_{\bark\,\barp}\,dz^\barl \wedge dz^\barj$$
in \eqref{K2} vanishes. Hence we obtain
\begin{eqnarray*}
\barpartial \barpartial_f^\ast (\varphi\lrcorner\,\omega)
 &=& \sqrt{-1} (\nabla_i + f_i)\nabla_\barl\,\varphi^i{}_{\barj} dz^\barl \wedge dz^\barj
 - \sqrt{-1} (R^\bark{}_\barl - f^\bark{}_\barl )\, \varphi_{\bark\,\barj} dz^\barl \wedge dz^\barj\\
 &=& \sqrt{-1}\diver_f(\barpartial \varphi) + \frac12\, \varphi\lrcorner\,\omega
\end{eqnarray*}
where we have used $\varphi_{\barl\barj}=\frac12(\varphi_{\barl\barj} + \varphi_{\barj\barl})
+ \frac12(\varphi_{\barl\barj} - \varphi_{\barj\barl})$.
This completes the proof of Lemma \ref{K1}.
\end{proof}

\begin{remark}
Similar formulae for the K\"ahler-Einstein manifold with $c_1 < 0$ and $c_1=0$ as in \eqref{Boch16}
have been obtained in \cite{Sun12}, \cite{SunYau11}. 
Later in \cite{ZhangYY2014}, it has been shown
$$
\Delta_f (\varphi\lrcorner\,\omega) = \sqrt{-1} \diver_f(\barpartial \varphi) + \frac12\,\varphi\lrcorner\,(\Ric - \nabla\bar\nabla f)
$$
under the additional assumption $\barpartial(\varphi\lrcorner\omega)=0$.
It was later used in \cite{CaoSunYauZhang2022}.
The full proof given here
and the the following observation would help to get better understanding of the arguments in the next section.
That is, the Ricci identity is used twice for \eqref{Boch11} with $q=2$ against the skew-symmetric indices $j_1$ and $j_2$ while the Ricci identity used
just once for \eqref{Boch16} because of the assumption $\barpartial_f^\ast\varphi = 0$.
\end{remark}

\section{K\"ahler forms of Kuranishi family}

First we briefly recall the deformation theory of complex structures by Kodaira, Spencer, Nirenberg and Kuranishi \cite{KodairaSpencer58},
\cite{KodairaNirenbergSpencer58}, \cite{Kuranishi64}, see also the monographs \cite{MorrowKodaira}, \cite{Kodaira86}.
Let $\varpi : \mathfrak M \to B$ be a complex analytic family of compact complex manifolds where $B$ is an open set in $\bfC^n$
containing the origin $0$.
Take a small neighborhood $N$ 
of the origin $0$ in $B$ with the coordinates
$(t^1, \cdots, t^n)$, and suppose that 
$\varpi^{-1}(N)$ is covered by coordinate neighborhoods as
\begin{eqnarray*}
\varpi^{-1}(N) &=& \cup_{\alpha \in A}\  U_\alpha,\\
 U_\alpha &=& \{(z_\alpha,t)\ |\ z_\alpha = (z_\alpha^1, \cdots, z_\alpha^m),\ |z_\alpha^i| \le 1,\ t \in N \}
\end{eqnarray*}
Over $U_\alpha \cap U_\beta \ne \emptyset$, if $(z_\alpha, t) \in U_\alpha$ and $(z_\beta,t) \in U_\beta$
are the same points then they are related by
$$ z_\alpha^i = f_{\alpha\beta}^i(z_\beta,t), \quad i=1, \cdots, m$$
for some holomorphic function $f_{\alpha\beta}^i$. The change of complex structures as $t$ varies is considered 
as the change of $f_{\alpha\beta}(z_\beta, t)$ in $t$. We put
$$ \theta_{\alpha\beta}(t) = \sum_i \frac{\partial f_{\alpha\beta}^i(z_\beta,t)}{\partial t}\frac{\partial}{\partial z_\alpha^i}$$
where
$$ \frac{\partial}{\partial t} = \sum_{\nu=1}^n c_\nu \frac{\partial}{\partial t^\nu}$$
is a tangent vector of $N$ at the origin.
Then $\{\theta_{\alpha\beta}(t)\}$ is a $1$-cocycle of the tangent sheaf $\Theta_t = \mathcal O(T^\prime M_t)$ of $M_t :=\varpi^{-1}(t)$. A different choice
of coordinates gives a cohomologous $1$-cocyle, and the cohomology class
$[\theta_{\alpha\beta}(t)] \in H^1(M_t, \Theta_t)$ is called the infinitesimal 
deformation of $M_t$ and denoted by $\frac{\partial M_t}{\partial t}$. Further, by the Dolbeault isomorphism $H^1(M_t,\Theta_t)
\cong H^{0,1}_\barpartial(M_t,T^\prime M_t)$, $\{\theta_{\alpha\beta}\}$ corresponds to a $\barpartial_t$-closed $T^\prime M_t$-valued $(0,1)$-form $\eta$. This correspondence is
described using the proof of Dolbeault theorem as follows. By the map $H^1(M_t,\Theta_t) \to H^1(M_t, \mathcal A^0(T^\prime M_t)) = 0$, there is a $1$-cochain 
$\{\xi_\alpha\}$ of $\mathcal A^0(T^\prime M_t)$ such that $\xi_\beta - \xi_\alpha = \theta_{\alpha\beta}$. Then $\eta = \barpartial_t \xi_\alpha$, which is 
independent of $\alpha$.

On the other hand, all $M_t = \varpi^{-1}(t)$ are diffeomorphic, and the complex structures on $M_t$ can be considered on the same differentiable manifold $M$.
Considering at $t$, we have a decomposition
$$ T^\ast M \otimes \bfC = T^{\prime\ast} M_t \oplus T^{\prime\prime\ast}M_t \quad \text{with}\quad T^{\prime\prime\ast}M_t = \overline{T^{\prime^\ast}M_t}.$$
Considering at $t=0$, for $t$ small, $T^{\prime\ast}M_t$ and $T^{\prime\prime\ast}M_t$ are close to $T^{\prime\ast}M_0$ and $T^{\prime\prime\ast}M_0$. In particular, $T^{\prime\ast}M_t$ is expressed as a graph over $T^{\prime\ast}M_0$ in the form that for all  $\bfx_t \in T^{\prime\ast}M_t$
$$\bfx_t = \bfx_0 + \varphi(t)(\bfx_0)$$
where $\bfx_0 \in T^{\prime\ast}M_0$ and $\varphi(t)(\bfx_0) \in T^{\prime\prime\ast}M_0$.
Thus 
\begin{eqnarray}
{}&&\varphi(t) \in \operatorname{Hom}(T^{\prime\ast}M_0,T^{\prime\prime\ast}M_0) = T^{\prime}M_0 \otimes T^{\prime\prime\ast}M_0, \nonumber\\
&& \quad \varphi(t) = \varphi^i{}_\barj (t) \frac{\partial}{\partial z^i} \otimes dz^\barj. \label{D1}
\end{eqnarray}
Here $z^i = z_0^i$ are local holomorphic coordinates of $M_0$, and we keep this notation below.
Thus $T^{\prime\ast}M_t$ is spanned by
\begin{eqnarray}\label{D2}
dz^i + \varphi^i{}_\barj(t) dz^\barj =: e^i, \qquad i = 1, \cdots, m,
\end{eqnarray}
or equivalently, $T^{\prime\prime}M_t$ is spanned by 
\begin{eqnarray}\label{D2.1}
\frac{\partial}{\partial z^\barj} - \varphi^i{}_\barj(t) \frac{\partial}{\partial z^i} =: T_\barj, \qquad i = 1, \cdots, m.
\end{eqnarray}
Then it can be shown as in \cite{MorrowKodaira}, page 152, Theorem 1.1, that
\begin{enumerate}
\item[(1)] $\varphi(0) = 0$;

\vspace{1.5mm}

\item[(2)] $\barpartial\varphi(t) - \frac12 [\varphi(t),\varphi(t)] = 0$;

\vspace{1.5mm}

\item[(3)] $\eta = - \frac{\partial \varphi(t)}{\partial t}|_{t=0}$.
\end{enumerate}
Here, for $T^\prime$-valued forms $\varphi$ and $\psi$, $[\varphi,\psi]$ is defined by the bracket with respect to $T^\prime$-components and
the wedge product with respect to the differential forms. Note that (2) is equivalent to the integrability condition of Newlander-Nirenberg.
Given $\barpartial$-closed $T^\prime$-valued $1$-form $\eta$ we wish to construct $\varphi(t)$ satisfying (1), (2) and (3).
We may choose an Hermitian metric on $M_0$ and take $\eta$ to be harmonic. Thus
$$ \barpartial \eta = 0 \quad \text{and} \quad \barpartial^\ast \eta = 0. $$
It is shown (c.f. \cite{MorrowKodaira}, Chapter 4, section 2) that  when $H^2(M,\Theta) = 0$ the power series expansion
$$ \varphi(t) = \sum_{|I|=1} t^I\varphi_I + \sum_{|I|\geq 2} t^I\varphi_I$$
has a unique solution under
\begin{enumerate}
\item[(4)] $\barpartial\varphi(t) = \frac12 [\varphi(t),\varphi(t)]$;

\vspace{1.5mm}

\item[(5)] $\barpartial^\ast \varphi(t) = 0$;

\vspace{1.5mm}

\item[(6)] For $|I| = 1$, $\varphi_I$ is harmonic.
\end{enumerate}
%When $H^2(M,\Theta) \ne  0$
%a small deformation of the complex structure on $M_0$ is caused by a diffeomorphsm of $M$, c.f. \cite{MorrowKodaira}, Chapter 4, section 3. 
But we apply this deformation theory
for Fano manifolds, and 
$$H^2(M_0,\Theta) \cong H^{m-2}(M_0, \Omega^1(K_{M_0})) =  0$$
by Serre duality and Kodaira-Nakano vanishing. Note that the 
condition (5) is equivalent to requiring
$$ \varphi(t) = \sum_{|I|=1} t^I\varphi_I + \frac12 \barpartial^\ast G [\varphi(t), \varphi(t)]. $$
Following \cite{Sun12}, \cite{CaoSunYauZhang2022}, we call the condition (5) the {\it Kuranishi gauge}.

Now let $M$ be a Fano manifold with a K\"ahler form $\omega$ in $2\pi c_1(M)$ with the Ricci potential $f$ as in \eqref{Boch10.5}.
In this case, 
one can argue as in \cite{ZhangYY2014}, Proposition 6, using a diffeomorphism, or 
instead of employing the Hodge theory using an Hermitian metric of the complex manifold $M_0$, we can employ the Hodge theory 
using the K\"ahler metric $\omega$ with weighted volume form $e^f \omega^m$, or equivalently using the K\"ahler metric $\omega$
and the bundle metric $e^f$ of the trivial line bundle $\mathcal O$ to apply the same arguments as in \cite{MorrowKodaira}, Chapter 4, section 2, and we
obtain the Kuranishi family described by
\begin{enumerate}
\item[(4)] $\barpartial\varphi(t) = \frac12 [\varphi(t),\varphi(t)]$.

\vspace{1.5mm}

\item[(7)] $\barpartial_f^\ast \varphi(t) = 0$.

\vspace{1.5mm}

\item[(8)] For $|I| = 1$, $\varphi_I$ is $\Delta_f$-harmonic.
\end{enumerate}
We call the condition (8) the {\it $f$-Kuranishi gauge}. 
This Kuranishi family was also considered in \cite{CaoSunZhang2022}. Then we can argue as in the proof of Theorem 2.1 in \cite{Sun12} to obtain the following lemma.
\begin{lemma}\label{D3}
For the Kuranishi family satisifying (4), (7), (8), we have 
$ \Delta_f (\varphi\lrcorner\,\omega) =  \frac12\, \varphi\lrcorner\,\omega$.
Combining this with Theorem \ref{ThmA}, we obtain $\varphi\lrcorner\,\omega =  0$.
\end{lemma}
\begin{proof}
We prove this by induction. First we consider the case $|I|=1$. By the condition (8), we know $\barpartial \varphi_I= 0$ and
$\barpartial_f^\ast \varphi_I= 0$. Then $\barpartial \varphi_I = 0$ implies $\barpartial (\varphi_I \lrcorner\,\omega) = 0$. It then follows from
Lemma \ref{K1} that 
$ \Delta_f (\varphi_I\lrcorner\,\omega )=   \frac12\,\varphi_I\lrcorner\,\omega$. 
Combining this with Theorem \ref{ThmA}, we obtain $\varphi_I \lrcorner\,\omega =  0$.
This completes the proof of the case $|I| = 1$. Note further that
$\varphi_I \lrcorner\,\omega =  0$ implies $\varphi_{I\bari\barj}$ is symmetric in $\bari$ and $\barj$. Thus $\barpartial_f^\ast \varphi_I= 0$ implies
$\diver_f \varphi_I = 0$.

Next we consider the case $|I| = k$ assuming $\varphi_J \lrcorner\,\omega =  0$ (and thus
$\diver_f \varphi_J = 0$) for $|J| \le k-1$. By Lemma \ref{K1} and (4) we have
\begin{eqnarray*}
 \bp (\varphi_I\lrcorner\, \omega) = (\bp\varphi_I)\lrcorner\, \omega 
= \frac12 \sum_{J+K=I}[\varphi_J,\varphi_K]\lrcorner\,\omega
= \sum_{J+K=I}\varphi_J\lrcorner\partial(\varphi_K\lrcorner\,\omega)
= 0
\end{eqnarray*}
since $\varphi_K\lrcorner\,\omega=0$ for $|K| < k$.
It follows from Lemma \ref{K1}, (7) and the above equality that 
\begin{eqnarray*}
\Delta_f (\varphi_I\lrcorner\,\omega ) &=& \sqrt{-1}\diver_f(\barpartial \varphi_I) + \frac12\,\varphi_I\lrcorner\,\omega\\
&=& \frac{\sqrt{-1}}2 \diver_f(\sum_{J+K=I}[\varphi_J,\varphi_K]) + \frac12\,\varphi_I\lrcorner\,\omega\\
&=& \sqrt{-1} \sum_{J+K=I}\varphi_J \lrcorner \partial (\diver_f\,\varphi_K) + \frac12\,\varphi_I\lrcorner\,\omega\\
&=& \frac12\,\varphi_I\lrcorner\,\omega.
\end{eqnarray*}
Combining this with Theorem \ref{ThmA}, we obtain $\varphi_I \lrcorner\,\omega =  0$.
This completes the proof of Lemma \ref{D3}.
\end{proof}

\begin{proof}[Proof of Theorem \ref{ThmD}]
Let $J_t$ be the complex structure of $M_t$. It is sufficient to show that $\omega$ is $J_t$-invariant. 
But this is equivalent to showing that $\omega$ vanishes on $T^\prime M_t$ or equivalently that 
$\omega$ vanishes on $T^{\prime\prime} M_t$. Since $\omega$ is a K\"ahler form for $J_0$ by the assumption and thus
vanishes on $T^\prime M_0$ and $T^{\prime\prime}M_0$, 
the description \eqref{D2.1} of $T^{\prime\prime}M_t$ shows that $\omega$ is $J_t$-invariant if and only if 
$\varphi(t)\lrcorner\,\omega = 0$ (c.f. the proof of Lemma 2.1 in \cite{futaki06} for some more detail). 
But this is indeed the case by Lemma \ref{D3}.
This completes the proof of Theorem \ref{ThmD}.
\end{proof}

\section{Ricci potential of the Kuranishi family}
In the previous section we showed that the K\"ahler form $\omega$ in $2\pi c_1(M_0)$ of a 
Fano manifold $M_0$ remains to be a K\"ahler form
on any member $M_t$ of the Kuranishi family of the deformations of the complex structure of $M_0$.
In this section we compute the Ricci potential of $(M_t, \omega)$. The result is Theorem \ref{ThmE} in the
Introduction.

To prove Theorem \ref{ThmE} we need preliminary calculations. In this paragraph (until the beginning of the proof of Theorem \ref{ThmE}),
we consider general deformations, not restricting to Fano manifolds, as we wish to re-use the equations in this 
paragraph in the next section. But in Proposition \ref{p6}, we assume that the K\"ahler form $\omega$ remains to be a K\"ahler form
for the deformations $M_t$ (or equivalently that $\omega$ is $J_t$-invariant for the complex structures $J_t$ of $M_t$), which we proved for a Fano manifold $M_0$ in Theorem \ref{ThmD} and we also assume in the next section
since it is a basic assumption in Donaldson-Fujiki picture.
Recall from \eqref{D2} and \eqref{D2.1}
that 
$T^{\prime\ast}M_t$ and $T^{\prime\prime}M_t$ are respectively spanned by
$$e^i = 
dz^i + \varphi^i{}_\barj(t) dz^\barj 
\quad \text{and} \quad 
T_\barj = \frac{\partial}{\partial z^\barj} - \varphi^i{}_\barj(t) \frac{\partial}{\partial z^i},\quad  
 i,\, j = 1, \cdots, m,$$ 
where $z^1, \cdots, z^m$ are local holomorphic coordinates for $M_0$.
We also use the notations $e_\bari := \overline{e_i}$ and $T_i := \overline{T_\bari}$. 
Let $w^1, \cdots, w^m$ be local holomorphic coordinates for $M_t$ on the same coordinate neighborhood as $z^i$'s.
Note that these two coordinates are related by
\begin{equation}\label{coor}
\frac{\partial w^\beta}{\partial z^\barj} = \varphi^i{}_\barj\,\frac{\partial w^\beta}{\partial z^i} \quad \text{and} \quad
\frac{\partial z^i}{\partial w^\barbet} = - \varphi^i{}_\barj\,\frac{\partial z^\barj}{\partial w^\barbet}
\end{equation}
from which we further obtain
\begin{equation}\label{p3}
dw^\alpha = \frac{\partial w^\alpha}{\partial z^i}\, e^i
 \quad \text{and} \quad
\frac{\partial}{\partial w^\beta} = \frac{\partial z^j}{\partial w^\beta}\,T_j.
\end{equation}
We put $\varphi := \left(\varphi^i{}_\barj\right)$ and $\barvarphi := \left(\overline{\varphi^i{}_\barj}\right)$. By computing
$\frac{\partial w^\alpha}{\partial w^\beta} = \delta^\alpha{}_\beta$ in terms of $z$ coordinates using \eqref{coor}, we obtain
\begin{equation}\label{p4}
\frac{\partial w^\alpha}{\partial z^i}\, (\delta^i{}_j - (\varphi\barvarphi)^i{}_j)\, \frac{\partial z^j}{\partial w^\beta} = \delta^\alpha{}_\beta.
\end{equation}
We put 
$$ A = \left(a^\alpha{}_i\right), \quad a^\alpha{}_i = \frac{\partial w^\alpha}{\partial z^i}, \quad A^{-1} = \left(b^i{}_\alpha\right) $$
so that $b^i{}_\alpha a^\alpha{}_j = \delta^i{}_j$. Then \eqref{p4} reads for sufficiently small $t$
\begin{equation}\label{p5}
A\,(I - \varphi\barvarphi)\,\frac{\partial z}{\partial w} = I, \quad \frac{\partial z^i}{\partial w^\alpha} = ((I - \varphi\barvarphi)^{-1})^i{}_j\,b^j{}_\alpha.
\end{equation}
Using \eqref{coor} we also have
\begin{equation}\label{p5.1}
\frac{\partial z^i}{\partial w^\barbet} = - \varphi^i{}_\barj\, \overline{((I - \varphi\barvarphi)^{-1})^j{}_\ell}\, \overline{b^\ell{}_\beta}.
\end{equation}
Expressing $\frac{\partial}{\partial w^\alpha}$ in terms of $z$ coordinates using \eqref{p3} and \eqref{p5} we also obtain
\begin{equation}\label{localframe}
\frac{\partial}{\partial w^\alpha} = ((I -\varphi\barvarphi)^{-1})^i{}_j\,b^j{}_\alpha\, T_i.
\end{equation}
Let $g_{\alpha\barbet}$ denote the K\"ahler metric of $(M_t,\omega)$ with respect to the local holomorphic coordinates $w^1, \cdots, w^m$,
namely,
$$ \omega = \sqrt{-1} g_{\alpha\barbet}\ dw^\alpha \wedge dw^\barbet.$$
We also retain the notation $g_{i\barj}$ to express the K\"ahler metric of $(M_0,\omega)$ with respect to the local holomorphic coordinates
$z^1, \cdots, z^m$ for $M_0$. We also use the notation $g_t = (g_{\alpha\barbet})$ and $g_0 = (g_{i\barj})$. 
\begin{proposition}\label{p6} If $\omega$ is a K\"ahler form for $M_t$, that is $\varphi\lrcorner\,\omega = 0$ then $g_t$ and $g_0$ are related by 
\begin{equation}\label{p7}
g_{\alpha\barbet}=\left( (I-\varphi\barvarphi\right)^{-1})^i{}_k \, b^k{}_{\ga}\,\overline{b^j{}_\beta}\, g_{i\bar j}. 
\end{equation}
\end{proposition}
\begin{proof}
Using the $z$ coordinates one compute
$$ -\sqrt{-1} \omega (T_i,T_\barj) = g_{i\barj} - \overline{\varphi^p{}_\bari}\,\varphi^q{}_\barj\,g_{\barp q}.$$
Using the assumption $\varphi \lrcorner\,\omega = 0$, i.e. $\varphi_{\bari\barj} = \varphi_{\barj\bari}$, we have
$$ \overline{\varphi^p{}_\bari}\,\varphi^q{}_\barj\,g_{\barp q} = \overline{\varphi_{\barq\bari}}\,\varphi^q{}_{\barj} = 
\overline{\varphi_{\bari\barq}}\,\varphi^q{}_{\barj} = g_{i\barl}\, \overline{\varphi^\ell{}_{\barq}}\,\varphi^q{}_{\barj}, $$
from which it follows that
\begin{equation}\label{p8}
-\sqrt{-1} \omega (T_i,T_\barj) = g_{i\barl}\,(I - \barvarphi\varphi)^\barl{}_\barj.
\end{equation}
Using \eqref{localframe} and \eqref{p8} we obtain
\begin{eqnarray*}
g_{\alpha\barbet} &=& ((I - \varphi\barvarphi)^{-1})^i{}_j\, b^j{}_\alpha\, ((I - \barvarphi\varphi)^{-1})^\bark{}_\barp\, \overline{b^p{}_\beta}\, 
g_{i\barl}\,(I - \barvarphi\varphi)^\barl{}_\bark\\
&=& ((I - \varphi\barvarphi)^{-1})^i{}_j\, b^j{}_\alpha\, \overline{b^\ell{}_\beta}\, g_{i\barl}.
\end{eqnarray*}
This is equal to \eqref{p7}. This completes the proof of Proposition \ref{p6}.
\end{proof}
\noindent
The following computations will be useful later. First, using \eqref{p3} we have
\begin{eqnarray}\label{c1}
\frac{\partial}{\partial w^\barbet} \log \det A &=& \frac{\partial z^\barj}{\partial w^\barbet}\, T_\barj \log \det A 
= \frac{\partial z^\barj}{\partial w^\barbet}\, b^i{}_\gamma\, T_\barj\, a^\gamma{}_i\\
&=& \frac{\partial z^\barj}{\partial w^\barbet} b^i{}_\gamma (\partial_\barj a^\gamma{}_i - \varphi^k{}_\barj \partial _k a^\gamma{}_i)\nonumber\\
&=& \frac{\partial z^\barj}{\partial w^\barbet} b^i{}_\gamma (\partial_i (\varphi^\ell{}_\barj a^\gamma{}_\ell) - \varphi^k{}_\barj \partial _k a^\gamma{}_i)\nonumber\\
&=& \frac{\partial z^\barj}{\partial w^\barbet} \partial_i\varphi^i{}_\barj,\nonumber
\end{eqnarray}
\noindent
and using \eqref{p3} and \eqref{p5} we have
\begin{eqnarray}\label{c2}
\frac{\partial}{\partial w^\barbet} \log \det \barA &=& \frac{\partial z^\barj}{\partial w^\barbet} \,T_\barj \log \det \barA 
= \frac{\partial z^\barj}{\partial w^\barbet}\, \overline{b^i{}_\gamma}\, T_\barj \,\overline{a^\gamma{}_i}\\
&=& \frac{\partial z^\barj}{\partial w^\barbet}\, \overline{b^i{}_\gamma}\, 
(\partial_\barj\, \overline{a^\gamma{}_i} - \varphi^k{}_\barj\, \partial _k\, \overline{a^\gamma{}_i})\nonumber\\
&=& \frac{\partial z^\barj}{\partial w^\barbet}\, \overline{b^i{}_\gamma}\, 
(\partial_\barj\, \overline{a^\gamma{}_i} - \varphi^k{}_\barj\, \overline{\partial_i(\varphi^\ell{}_\bark\,a^\gamma{}_\ell})\nonumber \\
&=& \overline{b^j{}_\beta}\, \overline{b^i{}_\gamma}\, \partial_\barj\overline{a^\gamma{}_i} - 
\frac{\partial z^\barj}{\partial w^\barbet} \varphi^k{}_\barj\,\overline{\partial_i\varphi^i{}_\bark}\,.\nonumber
\end{eqnarray}
Similarly, we can show
\begin{equation}\label{c3}
T_i\,\overline{b^j{}_\beta} = - \overline{b^k{}_\beta}\, \overline{\partial_k\varphi^j{}_\bari}
\end{equation}
and
\begin{equation}\label{c4}
T_i\,\overline{b^\ell{}_\beta\, b^j{}_\gamma\, \partial_\ell a^\gamma{}_j} = - \overline{b^\ell{}_\beta\, \partial_\ell \partial_j\, \varphi^j{}_\bari}\,.
\end{equation}

Now we are ready to prove Theorem \ref{ThmE}.
\begin{proof}[Proof of Theorem \ref{ThmE}]
By Proposition \ref{p6}
\begin{eqnarray*}
 \Ric(M_t,\omega) &=& -\partial_t\barpartial_t \log \det ((A\barA)^{-1} (I - \varphi\barvarphi))^{-1}\,g_0)\\
 &=& -\partial_t\barpartial_t \log ((\det ((A\barA)^{-1}\, g_0))e^f) + \partial_t\barpartial_t (f + \log\det(I - \varphi\barvarphi))\,.
 \end{eqnarray*}
 Thus, it is sufficient to show
 \begin{equation}\label{p9}
\partial_t\barpartial_t (\log (\det A \det \barA) - \log (e^f\det g_0)) = \omega.
\end{equation}
First of all
\begin{eqnarray}\label{p10}
\frac{\partial^2}{\partial w^\alpha \partial w^\barbet} \log(e^f\det g_0)
&=& \frac{\partial z^i}{\partial w^\alpha} T_i \left( \frac{\partial z^\barj}{\partial w^\barbet} \right) T_\barj \log (e^f \det g_0 ) \\
&& + 
\frac{\partial z^i}{\partial w^\alpha} \frac{\partial z^\barj}{\partial w^\barbet} T_i T_\barj \log (e^f \det g_0 ) \nonumber
\end{eqnarray}
Next, recall that the compatibility condition $\varphi\lrcorner\,\omega = 0$ and the $f$-Kuranishi gauge condition $\barpartial^\ast_f \varphi = 0$
implies $\diver_f\,\varphi = 0$. It follows from this that
\begin{equation}\label{p11}
\partial_i\varphi^i{}_\barj = - \varphi^i{}_\barj\, \partial_i \log(e^f\det g_0).
\end{equation}
Using \eqref{c1} and \eqref{p11} we obtain
\begin{eqnarray}\label{p12}
&&\quad \frac{\partial^2}{\partial w^\alpha \partial w^\barbet}\ \log \det A \\
&&=  - \frac{\partial z^i}{\partial w^\alpha} \left(T_i \frac{\partial z^\barj}{\partial w^\barbet}\right) \varphi^k{}_\barj\, \partial_k \log(e^f \det g_0)
  - \frac{\partial z^i}{\partial w^\alpha} \frac{\partial z^\barj}{\partial w^\barbet} T_i(\varphi^k{}_\barj \partial_k \log(e^f \det g_0)), \nonumber
\end{eqnarray}
and using \eqref{c2}, \eqref{p11} and \eqref{c3} we obtain
\begin{eqnarray}\label{p13}
%&& \\
&& \qquad \frac{\partial^2}{\partial w^\alpha \partial w^\barbet}\ \log \det \barA  \\
&&= \frac{\partial z^i}{\partial w^\alpha}  T_i( (\frac{\partial z^\barj}{\partial w^\barbet} - \overline{b^j{}_\beta}) \partial_\barj \log (e^f \det g_0) )
 - \frac{\partial z^i}{\partial w^\alpha}\, \overline{b^\ell{}_\beta}\, \partial_\barl \left( (\partial_\bark \log (e^f \det g_0))\overline{\varphi^k{}_\bari}) \right).\nonumber
\end{eqnarray}
It follows from \eqref{p10}, \eqref{p12}, \eqref{p13} and also \eqref{c3} that 
\begin{eqnarray*}%\label{p9}
\frac{\partial^2}{\partial w^\alpha \partial w^\barbet} \left( \log (\det A \det \barA) - \log (e^f\det g_0)\right) 
%&=& - \frac{\partial z^i}{\partial w^\alpha} \left(T_i \frac{\partial z^\barj}{\partial w^\barbet}\right) \varphi^k{}_\barj\, \partial_k \log(e^f \det g_0)
%  - \frac{\partial z^i}{\partial w^\alpha} \frac{\partial z^\barj}{\partial w^\barbet} T_i(\varphi^k{}_\barj \partial_k \log(e^f \det g_0)) \\
%&& + \frac{\partial z^i}{\partial w^\alpha}  T_i\left( \left(\frac{\partial z^\barj}{\partial w^\barbet} - \overline{b^j{}_\beta}\right) \partial_\barj \log (e^f \det g_0) \right)\\
%&& - \frac{\partial z^i}{\partial w^\alpha}\, \overline{b^\ell{}_\beta}\, \partial_\barl \left( (\partial_\bark \log (e^f \det g_0))\overline{\varphi^k{}_\bari}) \right)\\
%&&- \frac{\partial z^i}{\partial w^\alpha} T_i \left( \frac{\partial z^\barj}{\partial w^\barbet} \right) T_\barj \log (e^f \det g_0 )  
%- 
%\frac{\partial z^i}{\partial w^\alpha} \frac{\partial z^\barj}{\partial w^\barbet} T_i T_\barj \log (e^f \det g_0 ) \nonumber\\
&=&
 \frac{\partial z^i}{\partial w^\alpha} \overline{b^j{}_\beta}\,(R_{i\barj} - f_{i\barj})\\
 &=&  \frac{\partial z^i}{\partial w^\alpha} \overline{b^j{}_\beta}\,g_{i\barj}.
\end{eqnarray*}
It remains to show 
\begin{eqnarray*}
\sqrt{-1} \frac{\partial z^i}{\partial w^\alpha} \overline{b^j{}_\beta}\, g_{i\barj}\, dw^\alpha \wedge dw^\barbet
= \sqrt{-1}g_{i\barj}\, dz^i \wedge dz^\barj = \omega.
\end{eqnarray*}
But 
\begin{equation}\label{p14}
\sqrt{-1} \frac{\partial z^i}{\partial w^\alpha} \overline{b^j{}_\beta}\, g_{i\barj}\, dw^\alpha \wedge dw^\barbet
= \sqrt{-1} \, \overline{b^j{}_\beta}\, g_{i\barj}\, (dz^i - \frac{\partial z^i}{\partial w^\baralp}\,dw^\baralp)\wedge dw^\barbet.
\end{equation}
Since the compatibility condition $\varphi\lrcorner\,\omega = 0$ implies $\varphi_{\bark\,\barj} = \varphi_{\barj\,\bark}$,
using \eqref{p3},
the first term of the right hand side of \eqref{p14} is equal to
\begin{eqnarray*}%\label{p15}
\sqrt{-1} \, \overline{b^j{}_\beta}\, g_{i\barj}\, dz^i \wedge dw^\barbet
= \sqrt{-1}\, g_{i\barj}\, dz^i \wedge (dz^\barj +\overline{\varphi^j{}_\bark}\,dz^\bark )= \omega.
\end{eqnarray*}
The second term of the right hand side of \eqref{p14} vanishes because, using \eqref{p5} and \eqref{coor}, its coefficient is equal to
\begin{eqnarray*}
\sqrt{-1}\overline{(1-\varphi\barvarphi)^j{}_\ell\,\frac{\partial z^\ell}{\partial w^\beta}}g_{i\barj}\frac{\partial z^i}{\partial w^\baralp}
= \sqrt{-1}\lb \varphi_{\barp\barj}\barvarphi^{\barj\barq}\varphi_{\barq\barl}  - \varphi_{\barp\barl}\rb \overline{\frac{\partial z^\ell}{\partial w^\beta}\frac{\partial z^p}{\partial w^\alpha}}
\end{eqnarray*}
which is symmetric in $\alpha$ and $\beta$.
This completes the proof of Theorem \ref{ThmE}.
\end{proof}

\section{Scalar curvature as a moment map}

Let $(M, \omega)$ be a compact symplectic manifold, and suppose that $M_0=(M, J_0, \omega)$ is 
a \ka manifold with respect to an integrable complex structure $J_0$
where the K\"ahler metric $g_0$ is given by
$$ g_0(X,Y) = \omega (X, J_0 Y).$$
In this section, we consider
%we consider all integrable complex structures which are compatible with $\omega$. 
as in \cite{Donaldson84}
the space $\mathcal{J}_{int}(\omega)$ consisting of all $\omega$-compatible integrable complex structures $J$ so that
$(M, J, \omega)$ is a K\"ahler manifold.
%\begin{align}
%&\mathcal{J}_{int}(\omega) \\
%&=\{J\in \text{End}(TM)|\ J^2=-id, \ N_J=0, \ \omega(X, JY)+\omega(JX, Y)=0, \forall X, Y\in TM\}\nonumber\\
%\end{align}
For those complex structures $J$ obtained by deforming $J_0$ as in the Kodaira-Spencer theory, this space can be described as
an underlying real manifold of 
\begin{align}\label{icd}
\{\varphi\in A^{0,1}(M_0, T^\prime M_0)\ |\ \bar\pa_0\varphi=\frac{1}{2}[\varphi, \varphi], \ \varphi\lrcorner\,\omega=0\}.
\end{align}
The tangent space at $J_0$ is given by
\begin{align}\label{icd2}
T_{J_0}\mathcal{J}_{int}(\omega)=\{\psi\in A^{0,1}(M_0, T^\prime M_0)\ |\ \bar\pa_0\psi=0, \ \psi\lrcorner\omega=0\},
\end{align}
We consider the K\"ahler structure on $\mathcal{J}_{int}(\omega)$ induced from the K\"ahler structure on 
$A^{0,1}(M_0, T^\prime M_0)$, that is, the standard $L^2$ Hermitian inner product
\begin{align}
(\psi, \tau)_{L^2}=\int_{M_0} (\psi, \tau)_{g_0}\ \omega^m
\end{align}
for $\psi,\ \tau \in T_{J_0}\mathcal{J}_{int}(\omega)$ 
where,
in local holomorphic coordinates $(z^1, \dots, z^n)$ on $M_0$, if we write $g_0 = (g_{i\barj})$ and 
let $\psi=\psi^i{}_{\bar j}\,\frac{\pa}{\pa z^i}\otimes d\bar z^j$
and $\tau=\tau^i{}_{\bar j}\,\frac{\pa}{\pa z^i}\otimes d\bar z^j$, then
%and the \ka form $\omega_{J_0}=\sqrt{-1}g_{i\bar j}dz^i\wedge  d\bar z^j$, then
\begin{align*}
(\psi, \tau)_{g_0}&=\psi^i{}_{\bar j}\,\overline{\tau^k{}_{\barl}}\,g^{\ell\bar j}\,g_{i\bark}=\psi^i{}_\barj\,\overline{\tau_{\bari\,\barl}}\,g^{\ell\barj}
=\psi^i{}_\barj\,\overline{\tau_{\barl\,\bari}}\,g^{\ell\barj} =\psi^i{}_\barj\,\overline{\tau^j{}_{\bari}}=\tr(\psi\overline{\tau}).
\end{align*}
Twice the imaginary part $\Omega$ of $(\cdot, \cdot)_{L^2}$ gives a symplectic form on $\mathcal{J}_{int}(\omega)$, which is expressed as
\[\Omega(\psi, \tau)= 2\Im \int_{M_0}tr(\psi\overline{\tau})\ \omega^m.\]

For a tangent vector, i.e. an infinitesimal deformation, $\tau \in T_{J_0}\mathcal{J}_{int}(\omega)$, we consider a differentiable family $M_t$
with a real parameter $t$ in the direction of $\tau$ at $t=0$. Let $w^1, \cdots, w^m$ be local holomorphic coordinates for $M_t$ on the same coordinate neighborthood as $z^i$'s.
Then $w^\alpha$'s are smooth functions of $z^1, \overline{z^1}, \cdots, z^m, \overline{z^m}$ and $t$, and the equations \eqref{coor} through 
\eqref{c4} still hold true as was emphasized when the preliminary calculations were started in section 4. 
In particular, using \eqref{p3} and \eqref{localframe} we can express the complex structure $J_t$ of $M_t$ as
%We can express $J_t$ in terms of $z=(z^1, \dots, z^n)$ as
\begin{align}
J_t = & \sqrt{-1}\frac{\partial}{\partial w^\alpha}\otimes dw^\alpha - \sqrt{-1}  \frac{\partial}{\partial w^\baralp}\otimes dw^\baralp \label{h2}\\
=& \sqrt{-1}\lb((I-\varphi\barvarphi)^{-1})^i{}_j + \varphi^i{}_\bark\, \overline{((I-\varphi\barvarphi)^{-1})^k{}_\ell}\,\overline{\varphi^\ell{}_\barj}\rb\,
\frac{\partial}{\partial z^i }\otimes dz^j \nonumber\\
%(I+\bar\varphi\varphi)\ dz\otimes\frac{\pa}{\pa z}-2\sqrt{-1}(I-\bar\varphi\varphi)^{-1}\bar\varphi\ dz\otimes\frac{\pa}{\pa\bar z}\nonumber\\
&+ \sqrt{-1} \lb((I - \varphi\barvarphi)^{-1})^i{}_\ell\, \varphi^\ell{}_\barj\, 
%\frac{\partial}{\partial z^i }\otimes dz^\barj 
+ \sqrt{-1}\, \varphi^i{}_\bark ((I-\barvarphi \varphi)^{-1})^\bark{}_\barj\, \rb\frac{\partial}{\partial z^i }\otimes dz^\barj \nonumber\\
&- \sqrt{-1} \lb((I - \barvarphi\varphi)^{-1})^\bari{}_\barl\, \overline{\varphi^\ell{}_\barj}\, 
%\frac{\partial}{\partial z^\bari }\otimes dz^j 
- \sqrt{-1}\, \overline{\varphi^i{}_\bark}\, ((I-\varphi \barvarphi)^{-1})^k{}_j\rb\, \frac{\partial}{\partial z^\bari }\otimes dz^j \nonumber\\
&- \sqrt{-1}\lb((I-\barvarphi\varphi)^{-1})^\bari{}_\barj + \overline{\varphi^i{}_\bark}\, ((I-\varphi\barvarphi)^{-1})^k{}_\ell\,\varphi^\ell{}_\barj\rb\,
\frac{\partial}{\partial z^\bari }\otimes dz^\barj\,. \nonumber
\end{align}
The derivative of $J_t$ in \eqref{h2} is computed using infinitesimal deformation $\tau$ in the form of
\begin{align}\label{h3}
\left.\frac{d}{dt} \right|_{t=0} J_t = 2\sqrt{-1} \tau^i{}_\barj\, \frac{\partial}{\partial z^i }\otimes dz^\barj 
- 2\sqrt{-1}\, \overline{\tau^i{}_\barj}\,\frac{\partial}{\partial z^\bari }\otimes dz^j.
\end{align}

Consider the action of the Hamiltonian group $Ham(\omega)$ of $(M, \omega)$ on $\mathcal{J}_{int}(\omega)$. 
For a smooth function $u$ we denote by $X_u$ its Hamiltonian vector field. Then the infinitesimal action of $X_u$
on $\mathcal{J}_{int}(\omega)$ at $J_0$ is given by
\begin{align}\label{h1}
L_{X_u} J_0 = 2\sqrt{-1} \nabla_{J_0}^{\prime\prime} X_u^\prime - 2\sqrt{-1}\nabla_{J_0}^\prime X_u^{\prime\prime},
\end{align}
see e.g. Lemma 2.3 in \cite{futaki06}. It follows from \eqref{h1} and \eqref{h3} that the infinitesimal action of the Hamiltonian flow of generated by $X_u$ is expressed as an infinitesimal deformation as 
\begin{align}\label{h4}
\tau =  \nabla_{J_0}^{\prime\prime} X_u^\prime.
\end{align}
The Lie algebra of $Ham(\omega)$ is 
%It is known that the Lie algebra of $G$ is $Lie(G)= \{V\in \Gamma(TM)|\ i_v\omega=dh, \text{ for } h\in C^\infty(M)\}$. 
%Furthermore, this linear space can be identified with 
the space $C^\infty(M)/\bfR$ of smooth functions modulo constant functions, and its dual space is 
$$C_0^\infty(M) := \{u\in C^\infty(M)\ |\ \int_M u\ \omega^m =0\},$$
The Hamiltonian vector field $X_u$ of a smooth function $u$ is given by
$$ i(X_u)\omega = du. $$
This shows 
\begin{align}\label{h4.5}
\sqrt{-1}g_{i\barj}\,X_u^{\prime i}\, d z^\barj = \barpartial u,
\end{align}
which implies
$$ \tau^i{}_\barj = - \sqrt{-1} \nabla_\barj \nabla^i u.$$
It follows that the symplectic form $\Omega$ for the infinitesimal deformation $\tau$ induced by the Hamiltonian vector field $X_u$ becomes 
\begin{align}
\Omega(\psi,\tau) = 2\Re \int_{M_0} \psi^i{}_\barj\, \nabla^\barj\nabla_i u\ \omega^m 
= 2\Re \int_{M_0} \tr (\psi\cdot \overline{\barpartial\nabla^\prime u})\ \omega^m\,.
\end{align}

In the infinite dimensional GIT picture, Donaldson has shown that  

\begin{theorem} [Donaldson \cite{Donaldson84}]
The moment map of $Ham(\omega)$-action on $\mathcal{J}_{int}(\omega)$ is
\[s-\bar s:\mathcal{J}_{int}(\omega)\to C_0^\infty(M) \, (\cong Lie(Ham(\omega))^*),\]
where $s$ is the scalar curvature of $\omega_J$, and $\bar s$ is the average of $s$ on $M$. 
\end{theorem}

In the following, we will provide a proof of this theorem by using the Kodaira-Spencer theory.

\begin{proof}Since $\bar s$ is a topological constant, we only need to check that for any $u \in C_0^\infty(M)$, 
%let $X_u=\bar\pa_0\nabla^{1, 0}_{\omega_{J_0}}f\in T_{J_0}\mathcal{J}$, it holds that
%\begin{align}
%ds(f)=i_{X_f}\Omega.
%\end{align}
it is sufficient to show for any $\psi\in T_{J_0}\mathcal{J}$, 
\begin{align}\label{mpe}
\int_{M_0} \frac{d}{dt}\Big |_{t=0}s(J_t) u\ \omega^m = 2\Re \int_{M_0}\tr (\psi\cdot\overline{\bar\pa_0 \nabla^\prime u} )\ \omega^m,
\end{align}
where $J_t$ is the complex structure obtained by deforming $J_0$ along $\psi$-direction. We let $\varphi(t)$ be the corresponding Beltrami differential
%which depends on $t$ holomorphically according to the Kodaira-Spencer theory, 
with $\frac{d\varphi}{dt}|_{t=0}=\psi\in A^{0,1}(M_0, T^\prime M_0)$. 
Recall \eqref{p7}, from which we obtain
\begin{equation}\label{h5}
\det(g_{\ga\bar\gb})=|\det A|^{-2}\det(I-\varphi\barvarphi)^{-1}g_0,
\end{equation}
where $A=\det(a^\alpha{}_i)$, $g_0=\det (g_{i\bar j})$. From now on, we write $g$ instead of $g_0$ for the notational convenience. 
Recall also \eqref{p7} in Proposition \ref{p6}, from which we obtain
\begin{equation}\label{h6}
g^{\alpha\barbet} = a^{\alpha}{}_i\, (I - \varphi\barvarphi)^i{}_k\, g^{k\barj}\, \overline{a^\beta{}_j}.
\end{equation}

Then
the scalar curvature on $(X_t, \omega_t)$ is
\begin{align}\label{scalar}
s(t)&=-\Delta_t\log\det(g_{\ga\bar\gb})\\
&=\Delta_t\log A+\Delta_t\log\bar A+\Delta_t\log\det(I-\varphi\barvarphi)-\Delta_t\log g,\nonumber
\end{align}
where using \eqref{localframe}, \eqref{h6} and \eqref{c1}
\begin{align*}
\Delta_t&=g^{\ga\bar\gb}\frac{\pa^2}{\pa w^\ga\pa\bar w^\gb}\\
&= a^{\ga}{}_i\, (I - \varphi\barvarphi)^i{}_k\, g^{k\barj}\, \overline{a^\beta{}_j}
((I - \varphi\barvarphi)^{-1})^\ell{}_p \, b^p{}_\alpha\, T_\ell\, (\overline{((I-\varphi\barvarphi)^{-1})^p{}_q\, b^q{}_\beta}\,T_\barp\\
&= - g^{i\barj}(T_i\, \overline{a^\beta{}_j})\,\overline{((I-\varphi\barvarphi)^{-1})^p{}_q\, b^q{}_\beta} T_\barp
 + g^{i\barj}(T_i\overline{((I-\varphi\barvarphi)^{-1})^p{}_j})T_\barp \\
 &\ \quad + g^{i\barj}\,\overline{((I-\varphi\barvarphi)^{-1})^p{}_j}\, T_i\,\overline{T_\barp})\\
 &= - g^{i\barj}\,\overline{\partial_j\varphi^q{}_\bari}\,\overline{((I-\varphi\barvarphi)^{-1})^p{}_q}\, T_\barp
 + g^{i\barj}(T_i\overline{((I-\varphi\barvarphi)^{-1})^p{}_j})T_\barp \\
 &\ \quad + g^{i\barj}\,\overline{((I-\varphi\barvarphi)^{-1})^p{}_j}\, T_i\,\overline{T_\barp}).
\end{align*}
Since $\psi=\frac{d}{dt}\Big |_{t=0}\varphi$, then
\begin{align}\label{h6.1}
\frac{d}{dt}\Big |_{t=0}\Delta_t
=-g^{i\barj}\psi^k{}_{\bar j}\pa_i\pa_k-g^{i\bar j}(\pa_i\psi^k{}_{\bar j})\pa_k
-g^{i\barj}\overline{\psi^k{}_\bari}\pa_\barj\pa_\bark\, - g^{i\barj}(\overline{\pa_j\psi^k{}_\bari})\pa_\bark.
\end{align}
By \eqref{coor}, we compute
%and the fact that $w$-coordinate depends on $t$ holomorphically, 
\begin{equation}\label{h7}
 \frac{d}{dt}\Big |_{t=0}\Big(\frac{\pa w^\ga}{\pa\bar z^j}\Big)=\frac{\pa}{\pa\bar z^j}\Big(\frac{dw^\ga}{dt}\Big|_{t=0}\Big)=\,\delta^\ga{}_{i}\,\psi^i{}_\barj,
 %\quad \text {and} \quad
 %\frac{d}{dt}\Big |_{t=0}\Big(\frac{\pa\bar w^\ga}{\pa z^j}\Big)=0, 
\end{equation}
We also compute
\begin{equation}\label{h8}
\frac{d}{dt}\Big |_{t=0}\log \det A= \delta^k{}_{\ga}\pa_k\Big(\frac{dw^\ga}{dt}\Big|_{t=0}\Big).
%\quad \text{ and }\quad
%\frac{d}{dt}\Big |_{t=0}\log\bar A=\delta_{i\ga}\frac{d}{dt}\Big |_{t=0}\Big(\frac{\pa\bar w^\ga}{\pa\bar z^i}\Big)=0,
\end{equation}
By \eqref{h7} and \eqref{h8} we get
\begin{equation}\label{h9}
\Delta_0\left (\frac{d}{dt}\Big |_{t=0}\log \det A\right )=g^{i\bar j}\pa_i\pa_k\psi_{\bar j}^k.
%\quad \text{and }\quad
%\Delta_0\left (\frac{d}{dt}\Big |_{t=0}\log\bar A\right )=0.
\end{equation}
%Further computation shows that
%\begin{equation*}
%\begin{cases}
%\frac{d}{dt}\Big |_{t=0}\Big(\Delta_t\log A\Big)=\Delta_0\Big(\frac{d}{dt}\Big |_{t=0}\log A\Big)=g^{i\bar j}\pa_i\pa_k\psi_{\bar j}^k\\
%\frac{d}{dt}\Big |_{t=0}\Big(\Delta_t\log\bar A\Big)=\frac{d}{dt}\Big |_{t=0}\Big(\Delta_t\log\det(I-\bar\varphi\varphi)\Big)=0\\
%\frac{d}{dt}\Big |_{t=0}\Big(\Delta_t\log g\Big)=-g^{i\bar j}\psi_{\bar j}^k\pa_i\pa_k\log g-g^{i\bar j}\pa_i\psi_{\bar j}^k\pa_k\log g.
%\end{cases}
%\end{equation*}
Hence from , the linearization of the RHS of \eqref{scalar} is
\begin{align}
\frac{d}{dt}\Big |_{t=0}s(J_t)&=\Delta_0 \frac{d}{dt}\Big |_{t=0} (\log \det A + \log \det \barA) - (\frac{d}{dt}\Big |_{t=0} \Delta_t)\det g\\
%&= g^{i\bar j}\pa_i\pa_k\psi_{\bar j}^k+g^{i\bar j}\psi_{\bar j}^k\pa_i\pa_k\log g+g^{i\bar j}\pa_i\psi_{\bar j}^k\pa_k\log g \nonumber\\ 
&=\bar\pa_0^*div_0\psi + \partial_0^\ast \overline{\diver_0 \psi}.
\end{align}
Finally we have
\begin{align*}
\int_{M_0} \frac{d}{dt}\Big |_{t=0}s(J_t) u\ \omega^m&=\int_{M_0}(\barpartial_0^\ast\diver_0\psi + \partial_0^\ast \overline{\diver_0 \psi})\, u\, \omega^m\\
&= 2 \Re \int_{M_0} \tr\,\psi\ \overline{\bar\pa_0\nabla^\prime u} \ \omega^m\\
&=\Omega\left (\psi, \bar\pa_0\nabla^{\prime}u\right ).
\end{align*}
\end{proof}

\begin{remark}
%\begin{enumerate}
%\item 
% If we can take $f$ to be the scalar curvature $s$ of $\omega_{J_0}$ in \eqref{mpe}, then above computation shows that
%\begin{align}
%\frac{d}{dt}\Big |_{t=0}\int_{X_0} s(J_t)^2\ \frac{\omega^n}{n!}&=4\Omega(\psi, \bar\pa_0\nabla^{1, 0}s_0).
%\end{align}
%Hence $J_0$ is critical of the Calabi functional if and only if $\Omega(\psi, \bar\pa\nabla^{1, 0}s_0)=0$, i.e. $\bar\pa_0\nabla^{1, 0}s_0$ has to be in the orthogonal complement of $H^1(X_0, T^{1, 0}(X_0))$, but $\bar\pa_0\nabla^{1, 0}s_0$ apparently lies in $H^1(X_0, T^{1, 0}(X_0))$, hence  $\bar\pa_0\nabla^{1, 0}s_0$ must vanish, i.e. $\omega_{J_0}$ is an extremal metric. Hence, the extremal metric is a critical point of the Calabi functional in the space $\mathcal{J}$.
%\item
Let $P=\bar\pa_0\nabla^{\prime}: C^\infty_0(X_0)(\cong Lie(G))\to T_{J_0}\mathcal{J}$.
It was observed by Donaldson in \cite{Donaldson84} that if $\omega_{J_0}$ has constant scalar curvature, then
\begin{align}
T_{J_0}\mathcal{J}_{int}(\omega)\cong \ker P^*, \qquad P^*=\bar\pa_0^*div_0,
\end{align}
i.e.
\begin{align*}
T_{J_0}\mathcal{J}_{int}(\omega)&=\{\psi\in A^{0,1}(M_0, T^\prime M_0)|\ \bar\pa_0\psi=0, \ \psi\lrcorner\omega_{\text{csck}}=0\}\\
&=\{\psi\in A^{0,1}(M_0, T^\prime M_0)|\ \bar\pa_0\psi=0, \ \bar\pa_0^*div_0\psi=0\},
\end{align*}
where $\bar\pa_0^*$ and $div_0$ are with respect to the csck metric. 

On the other hand, if $\omega_0$ is \ke form,  It has been shown in \cite{CaoSunYauZhang2022} \cite{Sun12} \cite{SunYau11} that if $\bar\pa_0\varphi=\frac{1}{2}[\varphi, \varphi]$, then $\bar\pa_0^*\varphi=0$ if and only if $div_0\varphi=0$. Furthermore, $\varphi\lrcorner\omega_{\text{KE}}=0$ automatically holds. In particular this implies that the tangent space of the Kuranishi slice in the Kodaira-Spencer theory is a subspace of $T_{J_0}\mathcal{J}_{int}(\omega)$.
%\end{enumerate} 
\end{remark}

%We now briefly describe Donaldson's another moment map picture which fits in \ke geometry better. 
Next, we consider another moment map picture introduced by Donaldson \cite{Donaldson17} (see also \cite{LeeSturmWang})
on Fano manifolds, and describe it using our theory of the geometry of Kuranishi family developed above.  Namely, for a Fano manifold $M$, 
another Hermitian metric and a symplectic form on $\J:=\J_{int}\lb\omega\rb$ are introduced as an application of Theorem \ref{ThmD}, and 
consider the moment map for the Hamiltonian group action of $(M, \omega)$ where 
$\omega$ is a symplectic form on $M$ such that $\lsb \omega\rsb=2\pi c_1\lb M_J \rb$ with respect to some complex structure 
$J\in\J$ such that $M_J$ is a Fano manifold. 

Consider the Kuranishi family satsifying (4), (7) and (8). By Theorem \ref{ThmD}, $\omega$ is a K\"ahler form for any
member of the family.
We let $V_0=\int_M \lb 2\pi c_1\lb M\rb\rb^m$ and for each $J\in\J$ of the Kuranishi family 
we let $M_J$ be the corresponding Fano manifold. 
%For each $J\in\J_{int}(\omega)$ 
We also let $\Omega_J$ be the unique volume form on $M_J$ such that
\[
\begin{cases}
\rc_J\lb \Omega_J\rb=\omega\\
\int_{M_J}\Omega_J=V_0\,.
\end{cases}
\]
Let $f_J=\log\frac{\Omega_J}{\omega^m}$ be the normalized Ricci potential so that $\Omega_J = e^{f_J} \omega^m$. 

For any $J_0\in\J$ let $M_0:=M_{J_0}$ with local holomorphic coordinates $z^1,\cdots,z^m$, and let $\omega=\ii g_{i\bar j}dz^i\wedge d\bar z^j$ and $g=(g_{i\bar j})$.
%and we denote the Ricci potential by $f$ for $(M_0, \omega)$. 
But when we want to emphasize that $\omega$ and $g$ were expressed in terms of local holomorphic coordinates with respect to $J_0$ we
denote $\omega$ and $g$ by $\omega_0$ and $g_0$. 
We also write $f_0:=f_{J_0}$. 
A simple computation shows that $\bar\pa_0\dd_{f_0}\psi=0$ for each 
$$\psi\in T_{J_0}\J = \{\psi\in A^{0,1}(M_0, T^\prime M_0)\ |\ \bar\pa_0\psi=0, \ \psi\lrcorner\omega=0\}$$ 
(c.f. \eqref{icd2}) and there exists a unique complex valued smooth function $\xi_\psi$ such that
\begin{equation}\label{new0}
\begin{cases}
\bar\pa_0\xi_\psi=\dd_{f_0}\psi\\
\int_{M_0}\xi_\psi\ \Omega_0=0.
\end{cases}
\end{equation}

We define a Hermitian metric on $\J$ in the following way. For any $\psi,\tau \in T_{J_0}\J$ let
\begin{equation}\label{new}
\langle \psi,\tau\rangle=\int_{M_0}\lb \text{tr}\lb\psi\bar\tau\rb -\xi_\psi\bar{\xi_{\tau}}\rb\ \Omega_0.
\end{equation}
\begin{theorem}\label{ThmG} The Hermitian form \eqref{new} is positive definite.
\end{theorem}
\begin{proof}For the notational convenience we suppress $0$ and write $J$, $M$, $\barpartial$  and $f$ instead of $M_0$, $J_0$,
$\barpartial_0$ and $f_0$.
We need to show
\begin{equation}\label{new1}
\langle \psi,\psi\rangle=\int_{M}\lb |\psi|^2 -|\xi_\psi|^2\rb e^f \omega^m \ge 0,
\end{equation}
and that the equality holds if and only if $\psi = 0$.

Since $\barpartial \psi = 0$, by the Hodge theory, we have $\psi = \mu +\barpartial \tau$ where $\barpartial_f^\ast \mu = 0$ and 
$\tau \in A^0(T^\prime M)$. Then by Lemma \ref{K1} we have
$$\Delta_f (\mu\lrcorner\,\omega) = \frac12\,\mu\lrcorner\,\omega.$$
Then by Theorem \ref{ThmA} with $q=2$ we obtain 
$ \mu\lrcorner\,\omega = 0$. 
It follows from this and $\barpartial_f^\ast \mu = 0$ we have $\diver_f\mu = 0$. Hence we can assume $\psi = \barpartial \tau$. 
Put $\tau_\flat := g_{i\barj}\tau^i dz^\barj$. Then $\psi\lrcorner\, \omega = 0$ implies $\barpartial \tau_\flat = 0$. 
Since $M$ is Fano and $H^{0,1}(M) = 0$, there exists a complex valued smooth function $h$ such that
$\tau_\flat = \barpartial h$. Note that 
\begin{equation}\label{new2}
\diver_f\tau = -\Delta_f h.
\end{equation}
On the other hand, by a straightforward computation
\begin{equation}\label{new3}
(\diver_f\psi)_\barj = (\diver_f(\barpartial \tau))_\barj = \nabla_\barj(\diver_f \tau) + \tau_\barj = \nabla_\barj (\diver_f \tau + h).
\end{equation}
It follows from \eqref{new0}, \eqref{new2} and \eqref{new3} that
\begin{equation}\label{new4}
\xi_\psi = -\Delta_f h + h.
\end{equation}
Thus \eqref{new1} is equivalent to
\begin{equation}\label{new5}
\int_M (-\Delta_f h + h)^2 e^f \omega^m \le \int_M |\nabla^{\prime\prime}\nabla^{\prime\prime}h|^2 e^f \omega^m.
\end{equation}
Using the equality
\begin{equation*}
\int_M (\Delta_f h)^2 e^f \omega^m = \int_M (|\nabla^{\prime\prime}\nabla^{\prime\prime}h|^2 + |\nabla^{\prime\prime} h|^2) e^f \omega^m
\end{equation*}
one sees that \eqref{new5} is equivalent to
\begin{equation}\label{new6}
\int_M h^2 e^f \omega^m \le \int_M |\nabla^{\prime\prime}h|^2 e^f \omega^m.
\end{equation}
Note $\int_M h\, e^f\omega^m = 0$ as follows from \eqref{new0}. 
But \eqref{new6} holds by Corollary \ref{CorC}, and the equality occurs when $\Delta_f h = h$, that is, $\tau$ is a holomorphic
vector field. Thus $\psi = \barpartial \tau = 0$. 
This completes the proof of Theorem \ref{ThmG}.
\end{proof}

As above we denote by $G=Ham\lb M,\omega\rb$ with Lie algebra $\g\cong C^\infty\lb M, \bfR\rb/\bfR$. We can identify the dual $\g^*$ with the space of 
real 2m-forms 
$$ \{ \Omega - \omega^m \ |\ \int_M (\Omega - \omega^m) = 0\},$$
and the pairing $\lb\cdot,\cdot\rb:\g\times\g^*\to \R$ is given by $\lb h,\Omega - \omega^m\rb=\int_M h(\Omega - \omega^m)$ for any $h\in\g$ and $\Omega - \omega^m \in\g^*$. The group $G$ acts on $\J$. 
Minus of the twice the imaginary part of the Hermitian metric \eqref{new} defines a symplectic form, and using \eqref{h4} and \eqref{h4.5} the symplectic form
is expressed by
$$
%2\Re \int_{M_0} \text{Tr}\lb\psi\bar\tau\rb -\xi_\psi\bar{\xi_{\tau}}\rb\ \Omega_0
-2\Re \int_{M_0} (\tr (\psi\cdot\overline{\bar\pa \nabla^\prime u} ) - \xi_\psi\, (u-\Delta_{f_0} u))\ \Omega_0
$$
Donaldson showed the following. 
The original version of Donaldson's construction is different. It is easier to use this version to do local computations.%%
\begin{proposition}\label{new7}
The map $\mu:\J\to \g^*$ 
\[
\mu\lb J\rb=\Omega_J - \omega^m = e^{f_J}\omega^m - \omega^m,
\]
is the moment map for the action of $Ham(\omega)$, and a zero of $\mu$ gives a K\"ahler-Einstein Fano manifold.
\end{proposition}

\begin{proof}
For any $\psi\in T_{J_0}\J$, let us consider a differentiable family of deformations $\varphi(t)\in A^{0,1}\lb M_0, T^{\prime}M_0\rb$ 
with a real parameter $t$ such that 
\begin{enumerate}
\item $\bar\pa_0\varphi(t)=\frac 12 \lsb \varphi(t),\varphi(t)\rsb$.
%\item $\phi(t)$ depends on $t\in\Delta_\eps\subset\C$ holomorphically.
\item $\varphi(0)=0$ and $\dt\varphi(t)=\psi$. 
\end{enumerate}  
Then we claim 
\begin{equation}\label{new8}
\dt \mu\lb J_t\rb=2\Re\xi_\psi\,\Omega_0. 
\end{equation}
%\end{enumerate}
To prove this, first recall $\Ric(\Omega_t) = \omega_t$. Here $\omega_t$ denotes $\omega$ expressed in terms of local holomorphic 
coordinates $w^1, \cdots, w^m$ of the complex structure $J_t$. Thus,
$$ - \frac{\partial^2}{\partial w^\alpha \partial \overline{w^\beta}} \log \det (e^{f_t}g_t) = g_{\alpha\barbet}.$$
Then using \eqref{p3}, \eqref{p7} and \eqref{p4} we obtain
$$ T_i\lb\overline{\frac{\partial z^j}{\partial w^\beta}}\,T_\barj\,\log\det(e^{f_t}g_t)\rb = - \overline{b^j{}_\beta}\, g_{i\barj}.$$
Further using \eqref{p4} and \eqref{c3} we obtain
\begin{eqnarray}\label{new9}
&&\overline{\lb(I-\varphi\barvarphi)^{-1}\rb^j{}_k}T_iT_\barj\,\log\det(e^{f_t}g_t) 
+ T_i\overline{\lb(I-\varphi\barvarphi)^{-1}\rb^j{}_k}\,T_\barj\,\log\det(e^{f_t}g_t) \nonumber\\
&&\qquad\qquad -\ \overline{\partial_k\varphi^\ell{}_\bari} \overline{\lb(I-\varphi\barvarphi)^{-1}\rb ^j{}_\ell}\, T_\barj \log\det (e^{f_t}g_t) = - g_{i\bark}.
\end{eqnarray}
Taking the derivative of \eqref{new9} with respect to $t$ at $t=0$ we obtain
\begin{eqnarray}\label{new10}
&&\frac{\partial^2}{\partial z^i\partial z^\bark} \lb f^\prime(0) + \left.\frac{d}{dt}\right|_{t=0} \log \det g_t\rb \nonumber\\
&& = \frac{\partial}{\partial z^i}\lb \psi^p{}_\bark\,\frac{\partial}{\partial z^p}(\log\det e^{f_0}\,g_0)\rb
 +  \frac{\partial}{\partial z^\bark}\lb \overline{\psi^q{}_\bari}\,\frac{\partial}{\partial z^\barq}(\log\det e^{f_0}\,g_0)\rb.
\end{eqnarray}
Using \eqref{p7} we have
\begin{equation*}
\left.\frac{d}{dt}\right|_{t=0} \log \det g_t = - \delta^j{}_\alpha\,\frac{\partial}{\partial z^j}\lb\left.\frac{\partial w^\alpha}{\partial t}\right|_{t=0}\rb
- \overline{\delta^j{}_\alpha\,\frac{\partial}{\partial z^j}\lb\left.\frac{\partial w^\alpha}{\partial t}\right|_{t=0}\rb}\ ,
\end{equation*}
and further using \eqref{coor} (or \eqref{h7}) we obtain
\begin{equation}\label{new11}
\frac{\partial^2}{\partial z^i\partial z^\bark} \lb\left.\frac{d}{dt}\right|_{t=0} \log \det g_t\rb
= - \partial_i\partial_\ell \psi^\ell{}_\bark - \partial_\bark\partial_\barl \overline{\psi^\ell{}_\bari}.
\end{equation}
Then \eqref{new10} and \eqref{new11} imply that
\begin{eqnarray}\label{new12}
\frac{\partial^2}{\partial z^i\partial z^\bark} f^\prime(0) &=& \partial_i (\diver_{f_0}\psi)_\bark + \partial_\bark\overline{(\diver_{f_0}\psi)_\bari}\nonumber\\
&=& 2 \partial_i\partial_\bark \Re \xi_\psi.
\end{eqnarray}
Since $\int_{M_0} f^\prime(0) \Omega_0 = 0$ and $\int_{M_0} \xi_\psi \Omega_0 = 0$ we obtain from \eqref{new12}
$$ f^\prime(0) = 2 \Re \xi_\psi.$$
This completes the claim \eqref{new8}.
%%%%%%%%%%%%%
It follows that, the map $\mu$ being the moment map is equivalent to 
\[
\int_{M_0}\xi_\psi u\,\Omega_0=- \int_{M_0}\lb \tr (\psi\cdot\overline{\bar\pa \nabla^\prime u} ) - \xi_\psi (u-\Delta_{f_0}u)\rb\Omega_0.
\]
The right side of the above equation is 
\begin{align*}
\begin{split}
& - \int_{M_0}\lb \tr (\psi\cdot\overline{\bar\pa \nabla^\prime u} ) - \xi_\psi (u-\bar\Delta_{f_0}u)\rb\Omega_0\\
=& - \int_{M_0} \lb \psi_{\bar j}^i \pa_i\lb g^{k\bar j}\pa_k u\rb\rb \Omega_0+ \int_{M_0} \lb \xi_\psi\lb u-\bar\Delta_{f_0}u\rb\rb\Omega_0\\
=&  \int_{M_0}\lb \dd_{f_0}\psi\rb_{\bar j}g^{k\bar j}\pa_k u\,\Omega_0+ \int_{M_0}\xi_\psi\,u\,\Omega_0- \int_{M_0}\xi_\psi\bar\Delta_{f_0}u\,\Omega_0\\
=&  \int_{M_0}\pa_{\bar j}\xi_\psi g^{k\bar j}\pa_k u\, \Omega_0+ \int_{M_0}\xi_\psi u\,\Omega_0- \int_{M_0}\xi_\psi\bar\Delta_{f_0}u\,\Omega_0\\
=& \int_{M_0}\xi_\psi\bar\Delta_{f_0}u\,\Omega_0+ \int_{M_0}\xi_\psi u\,\Omega_0- \int_{M_0}\xi_\psi\bar\Delta_{f_0}u\,\Omega_0\\
=& \int_{M_0}\xi_\psi u\,\Omega_0\,. 
\end{split}
\end{align*}
This completes the proof of Proposition \ref{new7}.
\end{proof}

\section{Coupled K\"ahler metrics}
Let $M$ be a Fano manifold of complex dimension $m$. Suppose that we are given a decomposition of the anti-canonical bundle
$$ K_M^{-1} = L_1 \otimes \cdots \otimes L_k$$
into ample line bundles $L_\ga$'s. 
Let
$$ \omega_\ga = \sqrt{-1} g_{\ga i\barj}\, dz^i \wedge dz^\barj$$
be a K\"ahler form representing $2\pi c_1(L_\alpha)$. 
We denote by $\Delta_\ga$, $\nabla_\ga$ and $\Ric(g_\ga)$ the Laplacian, the covariant derivative and the Ricci curvature
of $g_\ga$ respectively.
Let $f_\alpha$ be the smooth function on $M$
such that 
$$ \Ric(g_\ga) = \omega_1 + \cdots + \omega_k + \sqrt{-1}\partial\barpartial f_\ga$$
with the normalization
$$ e^{f_1} \omega_1^m = \cdots = e^{f_k} \omega_k^m =: dV.$$
We call $\omega_1,\, \cdots,\ \omega_k$ {\it coupled K\"ahler metrics,} and
in this section we extend Theorem \ref{ThmA}, Theorem \ref{ThmB} and Corollary \ref{CorC} in this coupled setting.

As in \eqref{Boch11}, the Bochner-Kodaira formula reads in this setting
\begin{eqnarray} \label{Boch21}
(\Delta_{\ga, f_\ga}\,\eta)_{\barj_1\cdots\barj_q} &=& - g_{\ga}^{i\barj}(\nabla_{\ga, f_\ga})_i(\nabla_{\ga})_\barj\, \eta_{\barj_1\cdots\barj_q}\\
&& + \sum_{\beta=1}^q (R_{i \barj_\beta}(g_\ga) - f_{\ga\,i\barj_\beta})g_\ga^{i\barl}\,\eta_{\barj_1 \cdots \barj_{\beta -1}\barl\barj_{\beta+1}\cdots \barj_q}\nonumber\\
&=& - g_{\ga}^{i\barj}(\nabla_{\ga, f_\ga})_i(\nabla_{\ga})_\barj\, \eta_{\barj_1\cdots\barj_q}\nonumber\\
&& + \sum_{\beta=1}^q (\sum_{\gamma=1}^k g_{\gamma\, i \barj_\beta})\,g_\ga^{i\barl}\,\eta_{\barj_1 \cdots \barj_{\beta -1}\barl\barj_{\beta+1}\cdots \barj_q}
 \nonumber
\end{eqnarray}
where 
\begin{eqnarray} \label{Boch22}
(\nabla_{\ga,f})_i = \nabla_{\ga\,i} + f_{\ga\,i}.
\end{eqnarray}

For $(0,q)$-forms
$$ \eta_\alpha = \eta_{\alpha \barj_1 \cdots \barj_q} dz^{\barj_1} \wedge \cdots \wedge dz^{\barj_q}, \quad \alpha = 1, \cdots, k,$$
we put
$$ \eta_\alpha^\sharp := g_\alpha^{i\barj_1}\, \eta_{\alpha \barj_1 \cdots \barj_q} 
\frac{\partial}{\partial z^i}\otimes dz^{\barj_2} \wedge \cdots \wedge dz^{\barj_q}.$$
If we assume
\begin{equation}\label{Boch23}
\eta_1^\sharp = \cdots = \eta_k^\sharp := \psi,
\end{equation}
then from \eqref{Boch21} we have
\begin{eqnarray} \label{Boch24}
(\Delta_{\ga, f_\ga}\,\eta_\ga)_{\barj_1\cdots\barj_q}
= - g_{\ga}^{i\barj}(\nabla_{\ga, f_\ga})_i(\nabla_{\ga})_\barj\, \eta_{\ga \barj_1\cdots\barj_q}\nonumber 
+ q \sum_{\gamma=1}^k \eta_{\gamma\, \barj_1 \cdots \barj_q}.
\end{eqnarray}
From this we obtain the following Theorem \ref{ThmH}, Theorem \ref{ThmI} and Corollary \ref{CorJ} by similar proofs to Theorem \ref{ThmA},
Theorem \ref{ThmB} and Corollary \ref{CorC}.

\begin{theorem}\label{ThmH} (1)\ \ Suppose that \eqref{Boch23} holds. If 
\begin{equation}\label{Boch25}
\Delta_{\ga f_\ga}\,\eta_\ga = \lambda \sum_{\gamma=1}^k \eta_\gamma
\end{equation}
for some $\alpha$ 
%$$\Delta_{\ga\,f_\ga} \eta = \lambda \sum_{\beta=1}^q \eta^{\sharp(\ga\gb}\lrcorner\,\sum_{\gamma=1}^k$$
 and $\eta_\alpha \ne 0$ then $\lambda \ge q$.\\
(2)\ \ If, in (1), $\lambda = q$ then $\nabla_\ga^{\prime\prime}\eta_\ga = 0$. In particular $\barpartial \eta_\ga = 0$, and
also
$$ \eta_{\alpha \barj_1\cdots\barj_q} g_\ga^{i_1\barj_1} \cdots g_\ga^{i_q\barj_q} 
\frac{\partial}{\partial z^{i_1}}\wedge \cdots \wedge \frac{\partial}{\partial z^{i_q}} 
%= \psi^{i_1}{}_{\barj_2\cdots\barj_q} g^{i_2\barj_2} \cdots g^{i_q\barj_q} 
%\frac{\partial}{\partial z^{i_1}}\wedge \cdots \wedge \frac{\partial}{\partial z^{i_q}} 
$$
is a holomorphic section of $\wedge^q T^\prime M$. \\
%(3)\ \ Moreover, in (2), $\Delta_{\ga f_\ga}\,\eta_\ga = q \sum_{\gamma=1}^k \eta_\gamma$ is satisfied for 
%any other $\ga$, and the assertion of (2) holds for any other $\eta_\alpha$'s.
\end{theorem}

\begin{theorem}\label{ThmI} (1)\ \ Suppose that 
\begin{equation*}
(\barpartial\eta_1)^\sharp = \cdots =(\barpartial \eta_k)^\sharp
\end{equation*}
holds where 
$$ (\barpartial\eta_\alpha)^\sharp := g_\alpha^{i\barj_0}\, (\barpartial\eta)_{\alpha \barj_0 \cdots \barj_q} 
\frac{\partial}{\partial z^i}\otimes dz^{\barj_1} \wedge \cdots \wedge dz^{\barj_q}.$$
If 
\begin{equation}\label{Boch26}
\Delta_{\ga f_\ga}\,\eta_\ga = \lambda \sum_{\gamma=1}^k \eta_\gamma
\end{equation}
for some $\alpha$ 
%$$\Delta_{\ga\,f_\ga} \eta = \lambda \sum_{\beta=1}^q \eta^{\sharp(\ga\gb}\lrcorner\,\sum_{\gamma=1}^k$$
 and $\barpartial\eta_\alpha \ne 0$ then $\lambda \ge q+1$.\\
(2)\ \ If, in (1), $\lambda = q+1$ then $\nabla_\ga^{\prime\prime}\barpartial\eta_\ga = 0$. In particular
also
$$ (\barpartial\eta_{\alpha})_{\barj_0\cdots\barj_q}\, g_\ga^{i_0\barj_0} \cdots g_\ga^{i_q\barj_q} 
\frac{\partial}{\partial z^{i_0}}\wedge \cdots \wedge \frac{\partial}{\partial z^{i_q}} 
$$
is a holomorphic section of $\wedge^q T^\prime M$. 
\end{theorem}

\begin{corollary}[Theorem 1.2 in \cite{FutakiZhang18}]\label{CorJ}
If non-constant complex valued smooth functions $u_1, \dots, u_k$ satisfy
\begin{enumerate}
\item[(a)] $\grad^\prime_\ga u_\ga=\grad^\prime_\gb u_\gb$,\ \ $\ga, \gb=1, 2, \dots, k$
where
$$ \grad^\prime_\ga u_\ga = g^{i\barj}_\ga \nabla_\barj u_\ga \frac{\partial}{\partial z^i};$$
\item[(b)] $\Delta_{\ga, f_\ga} u_\ga=\lambda\sum\limits_{\gb=1}^k u_\gb$,\ \ for $\ga=1, 2, \dots, k$,
\end{enumerate}
then $\lambda\ge 1$. Moreover if $\lambda=1$, the complex vector field $V=\grad^\prime_\ga u_\ga=\grad^\prime_\gb u_\gb$ is a holomorphic vector field.\\
\end{corollary}

Of course Corollary \ref{CorJ} is the case of $q=0$ of Theorem \ref{ThmI}. It may be tempting to extend the result of \cite{CaoSunYauZhang2022}
for coupled K\"ahler-Einstein metrics introduced by \cite{HultgrenWittNystrom18}. We may use the Kuranishi gauge for one of the K\"ahler metrics,
say $g_1$, among $g_1,\ \cdots,\ g_k$, and it is possible to show that $\omega_1 + \cdots + \omega_k$ remains to be a K\"ahler form
for the members of the Kuranishi family using the case of $q=2$ of Theorem \ref{ThmH}. However, it is not clear whether the individual
$\omega_\alpha$ remains to be a K\"ahler form. Just as in the case of the volume minimization for transverse coupled K\"ahler-Einstein metrics
on Sasaki manifolds (c.f. \cite{futaki22}), straightforward extensions to the coupled case are not always possible.


\begin{thebibliography}{99}

\bibitem{CaoSunYauZhang2022} H.-D.~Cao, X.~Sun, S.-T.~Yau and Y.~Zhang : On deformations of Fano manifolds. Math. Ann. 383 (2022), no. 1-2, 809--836.

\bibitem{CaoSunZhang2022} H.-D.~Cao, X.~Sun and Y.~Zhang : A Weil-Petersson type metric on the space of Fano
K\"ahler-Ricci solitons.  J. Geom. Anal. 32 (2022), no. 12, Paper No. 288.

%\bibitem{CDS3}X.~X.~Chen, S.~K.~Donaldson, S.~Sun : K\"ahler-Einstein metric on Fano manifolds. III: limits with cone angle approaches $2\pi$ and completion of the main proof, J. Amer. Math. Soc. {\bf 28}, 235--278 (2015). 

\bibitem{Donaldson84}
  S.K.Donaldson.
\newblock{\em Remarks on gauge theory, complex geometry and 4-manifold topology. Fields Medallists' lectures, 384-403,}
 \newblock{World Sci. Ser. 20th Century Math., 5, World Sci. Publ., River Edge, NJ, 1997.}

\bibitem{Donaldson17}
S.K.Donaldson : 
The Ding functional, Berndtsson convexity and moment maps. Geometry, analysis and probability, 57--67, Progr. Math., 310, Birkh\"auser/Springer, Cham, 2017.

\bibitem{futaki83.1}A.~Futaki : 
An obstruction to the existence of Einstein K\"ahler metrics, Invent. 
Math. {\bf 73}, 437--443 (1983).

\bibitem{futaki87}A.~Futaki : The Ricci curvature of symplectic quotients of Fano manifolds. Tohoku Math. J., 39 (1987), no. 3, 329--339. 

\bibitem{futaki88}A.~Futaki : K\"ahler-Einstein metrics and integral invariants,
Lecture Notes in Math., vol.1314, Springer-Verlag, Berline-Heidelberg-New York,(1988).

\bibitem{futaki06}A.~Futaki :  Harmonic total Chern forms and stability,  Kodai Math. J. Vol. 29, No. 3 (2006),  346-369.

\bibitem{Futaki15}A.~Futaki : 
The weighted Laplacians on real and complex metric measure spaces, in Geometry and Analysis on Manifolds, In Memory of Professor Shoshichi Kobayashi, (eds. T.Ochiai et al), Progress in Mathematics, 
vol.308(2015), 343--351, Birkhauser.

\bibitem{futaki22}A.~Futaki : 
Moment polytopes on Sasaki manifolds and volume minimization.
To appear in Pure and Applied Mathematics Quarterly.
Preprint, arXiv:2201.10832


\bibitem{FutakiZhang18}A.~Futaki and Y.~Zhang : 
Coupled Sasaki-Ricci solitons. 
Science China Math., 64(2021), 1447--1462.

\bibitem{FutakiZhang19}A.~Futaki and Y.~Zhang : 
Residue formula for an obstruction to Coupled K\"ahler-Einstein metrics.
J. Math. Soc. Japan. 73(2021), 389--401.

\bibitem{HultgrenWittNystrom18}J.~Hultgren and D.~Witt Nystr\"om : Coupled K\"ahler-Einstein metrics.
Int. Math. Res. Not. IMRN 2019, no. 21, 6765--6796.

\bibitem{KJ}J.~L.~Kazdan : Gaussian and scalar curvature, an update, Seminar on
differential geometry (S.~T.~Yau, ed.) , Princeton Univ. Press, New
Jersey, 1982, 185-191.

\bibitem{KJ-FW1} J.~L.~Kazdan and F.~W.~Warner : Curvature functions for compact $2$-manifolds. 
Ann. of Math. (2) 99 (1974), 14--47.

\bibitem{Kodaira53}
K.~Kodaira : On a differential-geometric method in the theory of analytic stacks. Proc. Nat. Acad. Sci. U.S.A. 39 (1953), 1268--1273. 

\bibitem{Kodaira86}
K.~Kodaira : Complex manifolds and deformation of complex structures. Translated from the Japanese by Kazuo Akao. With an appendix by Daisuke Fujiwara. Grundlehren der mathematischen Wissenschaften [Fundamental Principles of Mathematical Sciences], 283. Springer-Verlag, New York, 1986. x+465 pp.

\bibitem{KodairaNirenbergSpencer58}
K. Kodaira, L.Nirenberg and D. C. Spencer : On the existence of deformations of complex analytic structures. Ann. of Math. (2) 68 (1958), 450--459.

\bibitem{KodairaSpencer58}
K. Kodaira and D. C. Spencer : On deformations of complex analytic structures. I, II, Ann.
of Math. (2), 67 (1958), pp. 328--466.



\bibitem{Kuranishi64} 
M.~Kuranishi : New proof for the existence of locally complete families of complex structures. 1965 Proc. Conf. Complex Analysis (Minneapolis, 1964) pp. 142--154 Springer, Berlin.

\bibitem{LeeSturmWang}
K.-L.~Lee, J.~Sturm  and X.~Wang, Moment map, convex function and extremal point. Preprint, arXiv:2208.03724.

\bibitem{MorrowKodaira} J.~Morrow and K.~Kodaira : Complex manifolds. Holt, Rinehart and Winston, Inc., New York-Montreal, Que.-London, 1971. vii+192 pp.

%\bibitem{Nakamura19}S.~Nakamura : Deformation for coupled K\"ahler-Einstein metrics, in preparation.


\bibitem{Sun12} X.~Sun : Deformation of canonical metrics I. Asian J. Math. 16 (2012), no. 1, 141--155.

\bibitem{SunYau11}
X.~Sun and S.-T.~Yau : Deformation of K\"ahler-Einstein metrics. Surveys in geometric analysis and relativity, 467-489, Adv. Lect. Math. (ALM), 20, Int. Press, Somerville, MA, 2011.

\bibitem{TakahashiJ} J.~Takahashi : On the gap between the first eigenvalues of the Laplacian on functions and 1-forms. J. Math. Soc. Japan 53 (2001), no. 2, 307--320.

%\bibitem{Tian12}G.~Tian: K-stability and K\"ahler-Einstein metrics. Comm. Pure Appl. Math. {\bf 68}, no. 7: 1085--1156 (2015).

%\bibitem{TianZhu02} G.~Tian and X.-H.~ Zhu : A new holomorphic invariant and uniqueness of K\"ahler-Ricci solitons, Comment. Math. Helv. 77(2002), 297-325.
 

 
% \bibitem{Wang-Zhu04} X.-J. Wang and X. Zhu : K\"ahler-Ricci solitons on toric manifolds with positive first Chern class, Adv. Math.  188  (2004),  no. 1, 87--103.


%\bibitem{yau78}S.-T.Yau : On the Ricci curvature of a compact K\"ahler manifold and the complex Monge-Amp\`ere equation I, Comm. Pure Appl. Math. {\bf 31}(1978), 339--441.

\bibitem{ZhangYY2014}
Y.~Zhang : Geometric Quantization of Classical Metrics on the Moduli Space of Canonical Metrics. Thesis (Ph.D.), Lehigh University. 2014. 85 pp. 

\end{thebibliography}
\end{document}